\documentclass[12pt,a4paper]{article}
\usepackage[a4paper,top=2.5cm,bottom=2.5cm,left=2.5cm,right=2.5cm]{geometry}
\usepackage{amsmath}
 \usepackage{graphicx}
\usepackage{amsthm}
\usepackage{color}
\usepackage{fancyhdr}
\usepackage{pdfpages}
\usepackage{txfonts}
\usepackage{amssymb}
\usepackage{paralist}
\usepackage{extarrows}
\usepackage{setspace}
\usepackage{braket} 
\usepackage[colorlinks=true,citecolor=blue,linkcolor=blue,anchorcolor=blue,urlcolor=blue]{hyperref}
\allowdisplaybreaks
\numberwithin{equation}{section}
\newtheorem{theorem}{Theorem}[section]
\newtheorem{lemma}{Lemma} [section]

\newtheorem{definition}[theorem]{Definition}

\allowdisplaybreaks
\begin{document}
\title{On the Yamabe flow in a bounded domain
 \footnote{This study was supported by the Key Program of the National Natural Science Foundation of China (Grant No. 12231016.)
 and the Program of the National Natural Science Foundation of China (Grant No. 12571117.)}}
\author{Fei Fang  \\  \footnotesize  \emph{School of Mathematics and Statistics,%
Beijing Technology and Business University, Beijing, 100048, China}
\\Zhong Tan\footnote{Corresponding author.  Email:  tanloy2026@126.com} \\  \footnotesize  \emph{School of Mathematical Science, Xiamen University, Xiamen, 361005, China}}
\maketitle
\noindent \textbf{\textbf{Abstract:}}
This paper investigates the dynamical behaviors of solutions to the Yamabe flow via the modified potential well method.
 We first establish the local existence and regularity of weak solutions for the flow.
 Several new results concerning global existence and blowup are obtained by classifying initial data into stable
 and unstable sets. Specifically, solutions with initial data in the stable set exist globally and extinguish in
  finite time, whereas those originating from unstable initial data blow up in infinite time.
   For certain high-energy initial data, we show that the solution decays to zero as time tends
   to infinity and undergoes finite-time blowup.
  In addition, we analyze Palais-Smale sequences to reveal the intrinsic relationship
  between the long-time asymptotic behavior of solutions and steady states. Finally, we derive the
  Pohozaev identity for the equation and prove the corresponding nonexistence theorem.



\section{Introduction}
Let $\Omega$ be a smooth bounded domain in $\mathbb{R}^N, N \geq 3$. We consider the following Yamabe flow
\begin{equation}\label{e01}
 \begin{cases}
& \dfrac{\partial}{\partial t} u^{p}=\Delta u+u^{p}, \  \text { in } \ \Omega \times(0, T), \\
&u(x,0)=u_0(x), \ \text { in } \  \Omega,\\
&  u=0, \ \text { on } \ \partial \Omega \times(0, T),
\end{cases}
\end{equation}
where $p=2^{\ast}-1=\frac{N+2}{N-2}$.
Problem \eqref{e01} possesses profound geometric and physical backgrounds.
Geometrically, consider the $N$-dimensional sphere $\mathbb{S}^N$ equipped with the
standard Riemannian metric $g_{\mathbb{S}^N}$. If the conformal metric evolves in the form
\[
g=v^{p-1}(\cdot, t) g_{\mathbb{S}^N},
\]
the corresponding Yamabe flow equation reads
\begin{equation}\label{ey02}
\left(v^{p}\right)_t=\Delta_{\mathbb{S}^N} v-c_N v, \quad c_N=\frac{N(N-2)}{4}.
\end{equation}
Via stereographic projection and cylindrical coordinate transformation, problem \eqref{e01} is equivalent to \eqref{ey02} in $\mathbb{R}^N$.
It has been established in \cite{rs9} and  \cite{rs42}  that the Yamabe flow \eqref{ey02} admits a
global solution converging exponentially to a steady state. For more research advances on the Yamabe flow,
readers may refer to \cite{rs8}, \cite{rs9}, \cite{rs10}, \cite{rs12}, \cite{rs13}, \cite{rs15}, \cite{rs16}, \cite{rs17}, \cite{rs40} and \cite{rs42}.

From a physical perspective, problem \eqref{e01} is closely related to the well-known fast diffusion equation formulated as
\begin{equation}\label{ey03}
\left\{
\begin{array}{l}
\dfrac{\partial w}{\partial \tau}=\Delta w^m \quad \text{in } \Omega \times(0, \infty), \\
w=0 \quad \text{on } \partial \Omega \times(0, \infty), \\
w(\cdot, 0)=w_0 \quad \text{in } \bar{\Omega},
\end{array}
\right.
\end{equation}
where $0<m<1$. Under the condition $mp=1$, we introduce the transformation
\[
u(x, t)=\left.(T-\tau)^{-m /(1-m)} w(x, \tau)^m\right|_{\tau=T\left(1-e^{-t}\right)}.
\]
Then the problem \eqref{ey03} can be converted into Equation \eqref{e01}. The problem \eqref{ey03} is a
typical nonlinear singular parabolic equation, which differs essentially from the porous medium equation corresponding to the case $m>1$.
problem \eqref{ey03} can be applied to describe various physical processes, such as the transport of
 high-temperature plasmas, gas migration in highly permeable porous media, and nonlinear diffusion of solutes in solutions.

A prominent property of solutions to \eqref{ey03} is the finite extinction phenomenon: the solution vanishes
identically over the entire domain after a finite time. This feature is in sharp contrast to the linear heat equation,
whose solution only decays asymptotically to zero as time tends to infinity.

Over the past decades, the asymptotic behavior of solutions near the extinction time has been a major research focus
in this field. For further relevant results, we refer the readers to
\cite{rs2}, \cite{rs3}, \cite{rs4}, \cite{rs6}, \cite{rs7}, \cite{rs13},
\cite{rs14}, \cite{rs26}, \cite{rs27}, \cite{rs28}, \cite{rs31}, \cite{rs35} and \cite{rs41}.

Payne and Sattinger \cite{Sattinger1} originally put forward the potential well method in the 1970s for solving
semilinear hyperbolic initial-boundary value problems. Later, Y.C. Liu et al. \cite{Liu1,Liu2} constructed a class of
parameter-dependent functionals and parameterized potential well sets, which greatly extended and optimized the original method.
Nowadays, the potential well method has evolved into a powerful approach to investigate the long-time dynamics
of solutions to evolution equations. Both the original method and its variants have been extensively applied to
 various equations and related problems, including semilinear parabolic equations \cite{tan},
 wave equations \cite{Vi}, fourth-order equations \cite{Han}, pseudoparabolic equations \cite{Xu},
  the heat flows of Dirichlet-to-Neumann operators \cite{Fang2},
  as well as heat flow problems associated with H-systems \cite{Fang1}.
At present, such techniques have emerged as one of the mainstream approaches for studying nonlinear
 problems involving critical and subcritical exponents.
 The main conclusions drawn in existing studies are summarized as follows:
\begin{description}
  \item[(1).]  If the initial value lies in the stable set, the solution exists globally and tends to zero asymptotically;
  \item[(2).] If the initial value belongs to the unstable set, the solution blows up in finite time.
\end{description}
Unlike previous works, for problem \eqref{e01}, we employ   the modified potential well method and obtain the following new results:
\begin{description}
  \item[(1).]  For initial values in the stable set, the solution exists globally and extinguishes in finite time;
  \item[(2).] For initial values in the unstable set, the solution blows up in infinite time;
  \item[(3).] For certain high-energy initial values, we prove two distinct phenomena: the solution either blows up in finite time or
            converges asymptotically over infinite time.
\end{description}
In addition, we establish the local existence and regularity of solutions. By virtue of the analysis of Palais-Smale sequences,
we also explore the relationship between the asymptotic behavior of solutions and steady states. Finally, we derive the
Pohozaev identity for the equation under consideration and prove the corresponding nonexistence theorem.

We  denote $\|\cdot\|_{L^q(\Omega)}$ by $\|\cdot\|_q$ for $1 \leq q \leq \infty$.
Let $X:=H_0^1(\Omega)$ with the norm
$$\|\nabla u\|_{L^2}=\|\nabla u\|_2=\left(\int_{\Omega} |\nabla u|^2dx\right)^{\frac{1}{2}}.$$
For $u \in X$, we set
$$
\begin{aligned}
E(u)&=\frac{1}{2} \int_{\Omega}|\nabla u|^2dx-\frac{1}{p+1} \int_{\Omega}|u|^{p+1} dx\\
&:=\frac{1}{2}A(u)-\frac{1}{2^{\ast}}B(u),\\
D(u)&=A(u)-B(u).
\end{aligned}$$
The Nehari manifold is defined by
$$
\mathcal{N}=\left\{u \in X: D(u)=0, u \neq 0\right\}.
$$
Correspondingly, we define two unbounded subsets
$$
\begin{aligned}
& \mathcal{N}_{+}=\left\{u\in X: D(u)>0\right\}, \\
& \mathcal{N}_{-}=\left\{u \in X: D(u)<0\right\}.
\end{aligned}
$$
The potential well and its corresponding set are defined, respectively, as
$$
\begin{aligned}
W & =\left\{u \in X: D(u)>0, E(u)<d\right\} \cup\{0\}, \\
V & =\left\{u \in X: D(u)<0, E(u)<d\right\},
\end{aligned}
$$
where
$$
d=\min _{v \in X \backslash\{0\}} \max _{t \geqslant 0} E(t u)=\inf _{v \in \mathcal{N}} E(u)
$$
is the depth of the potential well $W$. From the Sobolev inequality, we easily get
$$d=\frac{1}{N} S^{\frac{N}{2}},$$
where $S$ is the Sobolev constant (see \eqref{e03}).

Now let us define the level set
$$
E^\alpha=\left\{u \in X: E(u)<\alpha\right\} .
$$
Furthermore, by the definition of $E(u), \mathcal{N}, E^\alpha$ and $d$, we easily know that
$$
\mathcal{N}^\alpha=\mathcal{N} \cap E^\alpha \equiv\left\{u \in \mathcal{N}:
A(u)<\frac{2 \alpha(p+1)}{p-1}\right\} \neq \varnothing \quad \text { for all } \alpha>d.
$$
We now define
\begin{equation*}
\lambda_\alpha=\inf \left\{\| u\|_{p+1}: u\in \mathcal{N}^\alpha\right\}, \Lambda_\alpha=\sup \left\{\|u\|_{p+1}: u \in \mathcal{N}^\alpha\right\} \text { for all } \alpha>d .
\end{equation*}
It is clear that $\lambda_\alpha$ is nonincreasing and $\Lambda_\alpha$ is nondecreasing with respect to $\alpha$.

Let us define the modified functional and Nehari manifold as follows:
\begin{equation}\begin{aligned}
  D_{\delta}(u)&=\delta A(u)-B(u),\\
\mathcal{N}_\delta & =\left\{u \in X: D_{\delta}(u)=0,\|\nabla u\|_{2} \neq 0\right\}, \\
d_\delta & =\inf _{u \in N_\delta} E(u).
\end{aligned}\end{equation}
Therefore, we can define
\begin{equation}\label{e03}
 S=\inf \left\{\left.\frac{A(u)}{\|u\|_{p+1}^2} \right.: u \in X\setminus\{0\}\right\}
\end{equation}
and
\begin{equation}\label{e03ss}
r(\delta) =\delta^{\frac{N-2}{4}} S^{\frac{N}{4}}.
\end{equation}
 We also introduce the following sets:
$$
\begin{aligned}
\mathcal{B} & =\left\{u_0 \in X: \text { the solution } u=u(t) \text { of (\ref{e01}) blows up in finite time }\right\}, \\
\mathcal{G} & =\left\{u_0 \in X: \text { the solution } u=u(t) \text { of (\ref{e01}) exists for all } t>0\right\}, \\
\mathcal{G}_0 & =\left\{u_0 \in \mathcal{G}: u(t) \rightarrow 0 \text { in } X \text { as } t \rightarrow \infty\right\} .
\end{aligned}
$$
Then we can define the modified potential wells and their corresponding sets as follows:
\begin{align}
W_\delta & =\left\{u \in X: D_{\delta}(u)>0, E(u)<d(\delta)\right\} \cup\{0\},\nonumber \\
V_\delta & =\left\{u \in X: D_{\delta}(u)<0, E(u)<d(\delta)\right\}, \nonumber \\
B_\delta & =\left\{u \in X: \sqrt{A(u)}<r(\delta)\right\}, \nonumber \\
B_\delta^c & =\left\{u \in X:\sqrt{A(u)}>r(\delta)\right\} .
\end{align}

\section{Preliminaries and main lemmas}
In this section, we first define weak solutions and the maximal existence time, then introduce relevant notations and lemmas.
Throughout this paper, $C$ denotes a constant which may vary from line to line.
\begin{definition}\label{wd}
(Weak solution). We say that a function $u=u(x, t)$ is a weak solution of problem (\ref{e01}) in $\Omega_T:=\Omega \times(0, T)$ if and only if
$$
u \in L^{\infty}\left(0, T ; X\right), \frac{\partial}{\partial t}\left(u^{\frac{p+1}{2}} \right) \in L^2\left(\Omega_T\right)=L^2\left(0, T ; L^2(\Omega)\right),
$$
and satisfies problem (\ref{e01}) in the distribution sense, that is,
$$\left((u^p)_t, w\right)+\left(\nabla u, \nabla w\right) =\left(|u|^{p},w\right), \text { for all } w \in X, t\in(0,T),$$
where $u(x, 0)=u_0(x) \in X$.
\end{definition}

\begin{definition}\label{df1}
   (Maximal existence time). Assume that $u(t)$ is a weak solution of problem (\ref{e01}). The maximal existence time $T$ of $u(t)$ is defined as follows:
   \begin{description}
     \item[(1)] If $u(t)$ exists for $0 \leqslant t<\infty$, then $T=+\infty$.
     \item[(2)]  If there is a $t_0 \in(0, \infty)$ such that $u(t)$ exists for $0 \leqslant t<t_0$, but doesn't exist at $t=t_0$, then $T=t_0$.

   \end{description}
\end{definition}

\begin{lemma}\label{l03}Let $v \in X$. We obtain
  \begin{description}
    \item[(1)] If $0<\sqrt{A(u)}<r(\delta)$, then $D_{ \delta}(u)>0$. In particular, if $0<\sqrt{A(u)}<r(1)$, then $D(u)>0$;
    \item[(2)] If $D_{ \delta}(u)<0$, then $\sqrt{A(u)}>r(\delta)$. In particular, if $D(u)<0$, then $\sqrt{A(u)}>r(1) ;$
    \item[(3)] If $D_{\delta}(u)=0$, then $\sqrt{A(u)} \geqslant r(\delta)$ or $A(u)=0$.
    In particular, if $D(u)=0$, then $\sqrt{A(u)} \geqslant r(1)$ or $A(u)=0$;
    \item[(4)] If $D_{\delta}(u)=0$ and $A(u) \neq 0$, then $E(u)>0$ for $0<\delta<\frac{p+1}{2}, E(u)=0$ for $\delta=\frac{p+1}{2}, E(u)<0$ for $\delta>\frac{p+1}{2}$.
  \end{description}
\end{lemma}
\begin{proof}
 (1) Since $0<\sqrt{A(u)}<r(\delta)$, from  \eqref{e03} and \eqref{e03ss},   we obtain
$$\begin{aligned}
D_{\delta}(u) & =\delta A(u)-B(u) \\
& \geq \delta A(u)-\left(\frac{A(u)}{S}\right)^{\frac{p+1}{2}} \\
& \geq A(u)\left(\delta-\left(\frac{1}{S}\right)^{\frac{p+1}{2}}\left(A(u)\right)^{\frac{2}{N-2}}\right)>0.
\end{aligned}$$

(2) By the assumption $D_{\delta}(u)<0$ and \eqref{e03}, we have
 $$\begin{aligned}
0> D_{\delta}(u) & =\delta A(u)-B(u) \\
& \geq \delta A(u)-\left(\frac{A(u)}{S}\right)^{\frac{p+1}{2}} \\
& \geq A(u)\left(\delta-\left(\frac{1}{S}\right)^{\frac{p+1}{2}}\left(A(u)\right)^{\frac{2}{N-2}}\right).
\end{aligned}$$
Hence, $\sqrt{A(u)}>r(\delta)$.

(3) By the assumption $D_{ \delta}(u)=0$ and \eqref{e03}, we have

 $$\begin{aligned}
0= D_{\delta}(u) & =\delta A(u)-B(u) \\
& \geq \delta A(u)-\left(\frac{A(u)}{S}\right)^{\frac{p+1}{2}} \\
& \geq A(u)\left(\delta-\left(\frac{1}{S}\right)^{\frac{p+1}{2}}\left(A(u)\right)^{\frac{2}{N-2}}\right).
\end{aligned}$$
Hence, $\sqrt{A(u)}\geq r(\delta)$ or $u=0$.

(4) We easily know that

\begin{align}
 E(u) & =\frac{1}{2}A(u)-\frac{1}{p+1}B(u) \nonumber\\
 & =\left(\frac{1}{2}-\frac{\delta}{p+1}\right)A(u)+\frac{1}{p+1} (\delta A(u)-B(u))\nonumber\\
 & =\left(\frac{1}{2}-\frac{\delta}{p+1}\right)A(u)+\frac{1}{p+1} D_{\delta}(u) \nonumber\\
 & =\left(\frac{1}{2}-\frac{\delta}{p+1}\right)A(u) .
\end{align}
Then we can prove the conclusion.
\end{proof}

\begin{lemma}\label{l04}
\begin{description}
  \item[(1)]  $d(\delta) \geqslant a(\delta) r^2(\delta)$ for $a(\delta)=\frac{1}{2}-\frac{\delta}{p+1}, 0<\delta<\frac{p+1}{2}$;
  \item[(2)] $\lim _{\delta \rightarrow 0} d(\delta)=0, d\left(\frac{p+1}{2}\right)=0$ and $d(\delta)<0$ for $\delta>\frac{p+1}{2}$;
  \item[(3)] $d(\delta)$ is increasing on $0<\delta \leqslant 1$, decreasing on $1 \leqslant \delta \leqslant \frac{p+1}{2}$ and takes the maximum $d=d(1)$ at $\delta=1$.
\end{description}
\end{lemma}
\begin{proof}
(1) If $u \in \mathcal{N}_\delta$, by Lemma \ref{l03} (3), then $\sqrt{A(u)} \geqslant r(\delta)$. Moreover, we can deduce
\begin{equation} \label{e05}
\begin{aligned}
E(u) & =\frac{1}{2}A(u)-\frac{1}{p+1}B(v) \\
& =\left(\frac{1}{2}-\frac{\delta}{p+1}\right)A(u)+\frac{1}{p+1} D_{ \delta}(u) \\
& =\left(\frac{1}{2}-\frac{\delta}{p+1}\right)A(u) \geq a(\delta) r^2(\delta).
\end{aligned}
\end{equation}
Hence, $d(\delta) \geqslant a(\delta) r^2(\delta)$.

(2)
 We easily know that
$$
E(s u)=\frac{s^2}{2}A(u)-\frac{s^{p+1}}{p+1}B(u) .
$$
Hence,
\begin{equation} \label{e04}
\lim _{s \rightarrow 0} E(s u)=0.
\end{equation}
And if we let $s u \in \mathcal{N}_\delta$, then $s u$ satisfies

$$
0=D_{\delta}(s u)=\delta s^2A(u)-s^{p+1}B(u) .
$$
Then, we obtain
\begin{equation}\label{e06}
s(\delta)=\left(\frac{\delta A(u)}{B(u)}\right)^{\frac{1}{p-1}},
\end{equation}
which yields
\begin{equation*}
\lim _{\delta \rightarrow 0} s(\delta)=0.
\end{equation*}
Now (\ref{e04}) implies that
\begin{equation*}
\lim _{\delta \rightarrow 0} E(s u)=\lim _{\lambda \rightarrow 0} E(s u)=0,
\end{equation*}
and
\begin{equation*}
\lim _{\delta \rightarrow 0} d(\delta)=0.
\end{equation*}
It is easy to see that from (\ref{e05})
$$
d\left(\frac{p+1}{2}\right)=0 \text { and } d(\delta)<0 \text { for } \delta>\frac{p+1}{2} .
$$

(3) We need to prove that for any $0<\delta^{\prime}<\delta^{\prime \prime}<1$
or $1<\delta^{\prime \prime}<\delta^{\prime}<\frac{p+1}{2}$ and for
any $w \in \mathcal{N}_{\delta^{\prime \prime}}$, there is a
 $v \in \mathcal{N}_{\delta^{\prime}}$ and a constant $\varepsilon\left(\delta^{\prime}, \delta^{\prime \prime}\right)$
 such that $E(v)<E(w)-\varepsilon\left(\delta^{\prime}, \delta^{\prime \prime}\right)$.
 Indeed, by  (\ref{e06}), we easily know that $D_{ \delta^{\prime \prime}}(s(\delta^{\prime \prime}) w)=0$
 and $s\left(\delta^{\prime \prime}\right)=1$. Let $h(s)=E(s w)$
  we have
$$
\begin{aligned}
\frac{\mathrm{d}}{\mathrm{d} s} h(s) & =sA(w)-s^pB(w) \\
& = sA(w)-\delta^{\prime\prime} s^p A(w)+s^pD_{\delta^{\prime\prime}}(w)\\
& = (s-\delta^{\prime\prime} s^p) A(w)+s^pD_{\delta^{\prime\prime}}(w)\\
& = s(1-\delta^{\prime\prime} s^{p-1}) A(w).
\end{aligned}
$$
Take $v=s\left(\delta^{\prime}\right) w$, then $v \in \mathcal{N}_{\delta^{\prime}}$.
For $0<\delta^{\prime}<\delta^{\prime \prime}<1$,  $s\in (s(\delta^{\prime}), s(\delta^{\prime\prime}))=(s(\delta^{\prime}), 1)$, we obtain
$$
\begin{aligned}
E(w)-E(v) & =h(1)-h\left(s\left(\delta^{\prime}\right)\right) \\
& >\left(1-\delta^{\prime \prime}\right) r^2\left(\delta^{\prime \prime}\right) s\left(\delta^{\prime}\right)\left(1-s\left(\delta^{\prime}\right)\right) \equiv \varepsilon\left(\delta^{\prime}, \delta^{\prime \prime}\right).
\end{aligned}
$$
For $1<\delta^{\prime \prime}<\delta^{\prime}<\frac{p+1}{2}$, we obtain
$$
\begin{aligned}
E(w)-E(v) & =h(1)-h\left(s\left(\delta^{\prime}\right)\right) \\
& >\left(\delta^{\prime \prime}-1\right) r^2\left(\delta^{\prime \prime}\right) s\left(\delta^{\prime \prime}\right)\left(s\left(\delta^{\prime}\right)-1\right) \equiv \varepsilon\left(\delta^{\prime}, \delta^{\prime \prime}\right).
\end{aligned}
$$
Hence, the proof is complete.
\end{proof}

\begin{lemma}\label{l06}
Let $u \in X$ and $0<\delta<\frac{p+1}{2}$. If $E(u) \leqslant d(\delta)$, then we have
 \begin{description}
   \item[(1)] If $D_{\delta}(u)>0$, then $A(u)<\frac{d(\delta)}{a(\delta)}$, where $a(\delta)=\frac{1}{2}-\frac{\delta}{p+1}$.
   In particular, if $E(u) \leqslant d$ and $D(u)>0$, then
\begin{equation*}
A(u)<\frac{2(p+1)}{p-1} d;
\end{equation*}
   \item[(2)] If $A(u)>\frac{d(\delta)}{a(\delta)}$, then $D_{ \delta}(u)<0$. In particular, if $E(u) \leqslant d$ and
\begin{equation*}
A(u)>\frac{2(p+1)}{p-1} d;
\end{equation*}
then $D(u)<0$.
   \item[(3)] If $D_{ \delta}(u)=0$, then $A(u) \leqslant \frac{d(\delta)}{a(\delta)}$. In particular, if $E(u) \leqslant d$ and $D(u)=0$, then
\begin{equation*}
A(u) \leqslant \frac{2(p+1)}{p-1} d.
\end{equation*}
\end{description}
\end{lemma}
\begin{proof}
(1) For $0<\delta<\frac{p+1}{2}$, we see that
\begin{equation}\label{e07}
\begin{aligned}
E(u) & =\frac{1}{2}A(u)-\frac{1}{p+1}B(u) \\
& =\left(\frac{1}{2}-\frac{\delta}{p+1}\right)A(u)+\frac{1}{p+1} D_{\delta}(u) \\
& =a(\delta)A(u)\leq d(\delta).
\end{aligned}
\end{equation}
Therefore,
$$
A(u)<\frac{d(\delta)}{a(\delta)} .
$$
Finally, (2) and (3) follow from (\ref{e07}).
\end{proof}

\begin{lemma}\label{l07}
Let $u \in X$. We have
\begin{description}
  \item[(1)] 0 is away from both $\mathcal{N}$ and $\mathcal{N}_{-}$,
  i.e. $\operatorname{dist}(0, \mathcal{N})>0, \operatorname{dist}\left(0, \mathcal{N}_{-}\right)>0$;
  \item[(2)]  For any $\alpha>0$, the set $E^\alpha \cap \mathcal{N}_{+}$ is bounded in $X$.
\end{description}
\end{lemma}
\begin{proof}
 (1) If $u \in \mathcal{N}$, then we have
\begin{align}\label{jiu}
d & \leq E(u)=\frac{1}{2}A(u)-\frac{1}{p+1}B(u)\nonumber \\
& =\left(\frac{1}{2}-\frac{1}{p+1}\right)A(u)+\frac{1}{p+1} D(u) \\
& =\left(\frac{1}{2}-\frac{1}{p+1}\right)A(u).\nonumber
\end{align}
If $u \in \mathcal{N}_{-}$, then we have
$$
\begin{aligned}
d & \leq E(u)=\frac{1}{2}A(u)-\frac{1}{p+1}B(u) \\
& =\left(\frac{1}{2}-\frac{1}{p+1}\right)A(u)+\frac{1}{p+1} D(u) \\
& \leq\left(\frac{1}{2}-\frac{1}{p+1}\right)A(u).
\end{aligned}
$$
Hence, 0 is away from both $\mathcal{N}$ and $\mathcal{N}_{-}$, i.e. $\operatorname{dist}(0, \mathcal{N})>0$, $\operatorname{dist}\left(0, \mathcal{N}_{-}\right)>0$.

(2) Since $E(u)<\alpha$ and $D(u)>0$, we obtain

$$
\begin{aligned}
\alpha & >E(u)=\frac{1}{2}A(u)-\frac{1}{p+1}B(v) \\
& =\left(\frac{1}{2}-\frac{1}{p+1}\right)A(u)+\frac{1}{p+1} D(u) \\
& >\left(\frac{1}{2}-\frac{1}{p+1}\right)A(u).
\end{aligned}
$$
Hence, for any $\alpha>0$, the set $E^\alpha \cap \mathcal{N}_{+}$ is bounded in $X$.
\end{proof}

\section{Local existence and regularity}
In this section, we investigate the local existence and regularity of weak solutions to equation \eqref{e01}.
\begin{theorem}\label{t21}
  Suppose that $u_0 \in  H_0^1(\Omega)$. Then there exists a weak solution $u$ to problem \eqref{e01} such that
  $u \in L^\infty\left(0, T; H_0^1(\Omega)\right)$, $\frac{\partial}{\partial t}\left(u^{\frac{p+1}{2}} \right)\in  L^2\left(0, T; L^2(\Omega)\right)$.
\end{theorem}

\begin{theorem}\label{t22}
 Let $u$ be a global solution with  initial value $u_0(x)\in W$.
 Then $u \in L^q\left(\Omega \times\left(t_0, \infty\right)\right)$ for any $q$ $ (2 \leq q<\infty)$ and any $t_0>0$, and
$$
\|u\|_{L^q\left(\Omega \times\left(t_0, \infty\right)\right)} \leq C,
$$
where  $C$ depends only on $N, q$, and $t_0$. In particular, $u$ is a classical solution of problem (\ref{e01}) for all $t \geq t_0>0$.
\end{theorem}

\begin{proof}[Proof of Theorem \ref{t21}]
We first construct approximate solutions to problem \eqref{e01}. Define
\begin{equation*}
\varphi_k^{+}(s)=
\begin{cases}
p s^{p-1}, & s \geq k^{-1}, \\
p\left(a_k s^2+b_k\right), & 0\leq s<k^{-1},
\end{cases}
\end{equation*}
where
\[
a_k=\frac{k^{2-p}(p-1)}{2},\quad b_k=\frac{k^{1-p}(3-p)}{2},\quad k=1,2,\ldots.
\]
We extend $\varphi_k^{+}(s)$ even-symmetrically to the entire real line $\mathbb{R}$ and denote the extended function by $\varphi_k(s)$. It is easy to verify that $\varphi_k(s) \in C^1(\mathbb{R})$ and
\[
\varphi_k(s) \to p|s|^{p-1} \quad \text{as } k \to \infty \text{ for all } s \neq 0.
\]
Since $u_0\in X$, we may choose a sequence $\{u_{0,k}\} \subset C_0^2(\Omega)$ satisfying
\[
\big\|\nabla u_{0,k}-\nabla u_0\big\|_{2} \to 0 \quad \text{as } k \to \infty.
\]
Now we consider the following approximate problem:
\begin{equation}\label{ebi}
\begin{cases}
\varphi_k(u)u_t =\Delta u+u^{p}, & \text{in } \Omega \times(0, T), \\
u(x,0)=u_{0,k}(x), & \text{in } \Omega,\\
u=0, & \text{on } \partial \Omega \times(0, T).
\end{cases}
\end{equation}
One can readily check that problem \eqref{ebi} satisfies all the assumptions for global classical solvability
established by Lady\v{z}eskaya et al. (Theorem 4.1, p. 558 in \cite{rlady}). Consequently, problem \eqref{ebi} admits a unique nonnegative classical solution $u_{k}(x,t)$.

Multiplying both sides of \eqref{ebi} by $u_{k,t}$ and integrating by parts, we derive
\begin{equation}\label{ebi1}
\int_{\Omega}\varphi_k(u_{k})u_{k,t}^2\mathrm{d}x
+\frac{\mathrm{d}}{\mathrm{d}t}\left(\frac{1}{2}\int_{\Omega} |\nabla u_k|^2\mathrm{d}x\right)
=\int_{\Omega }u_k^{p}u_{k,t}\mathrm{d}x
\end{equation}
and
\begin{equation}\label{ebi2}
\int_0^t\int_{\Omega}\varphi_k(u_{k})u_{k,t}^2\mathrm{d}x\mathrm{d}\tau+E(u_k(t))=E(u_k(0)),\quad t\in[0,T].
\end{equation}

We apply the weighted Young inequality
\[
ab \leq \frac{\varepsilon}{2} a^2+\frac{1}{2 \varepsilon} b^2,\quad \forall \varepsilon>0,
\]
to the product term $|u_{k,t}| \cdot |u_k^p|$. Set
\[
a=\sqrt{\varphi_k(u_k)} \cdot |u_{k,t}|, \quad b=\frac{|u_k|^p }{\sqrt{\varphi_k(u_k)}}.
\]
Then
\[
|u_{k,t}| \cdot |u_k|^p
=\left(\sqrt{\varphi_k(u_k)}\,|u_{k,t}|\right) \cdot \frac{|u_k|^p}{\sqrt{\varphi_k(u_k)}}
=ab.
\]
Substituting this into the inequality and integrating over $\Omega$ yields
\begin{equation}\label{ebi5}
\int_{\Omega }|u_k^{p}u_{k,t}|\mathrm{d}x
\leq \frac{\varepsilon}{2}\int_{\Omega } \varphi_k(u_k) u_{k,t}^2\mathrm{d}x
+\frac{1}{2 \varepsilon} \int_{\Omega } \frac{|u_k|^{2 p}}{\varphi_k(u_k)}\mathrm{d}x.
\end{equation}
Define two subdomains
\[
\Omega_1:=\big\{x \in \Omega \,\big|\, | u_k(x, t) | \geq k^{-1}\big\}, \quad
\Omega_2:=\Omega\setminus\Omega_1=\big\{x \in \Omega \,\big|\, | u_k(x, t) | \leq k^{-1}\big\}.
\]
By the definition of $\varphi_k$, we obtain
\begin{equation}\label{ebi4}
\begin{gathered}
\int_{\Omega_1}\frac{|u_k|^{2 p}}{\varphi_k(u_k)} \mathrm{d}x \leq \int_{\Omega}|u_k|^{p+1} \mathrm{d}x \triangleq I, \\
\int_{\Omega_2}\frac{|u_k|^{2 p}}{\varphi_k(u_k)} \mathrm{d}x
\leq k^{p-1}\int_{\Omega_2}|u_k|^{2 p} \mathrm{d}x
\leq C|\Omega|k^{-p-1} \triangleq II.
\end{gathered}
\end{equation}
Combining \eqref{ebi1}, \eqref{ebi5}, \eqref{ebi4} and the Sobolev inequality, we deduce
\begin{align*}
&\int_{\Omega} \varphi_k\left(u_k\right) u_{k, t}^2 \mathrm{d}x
+\frac{\mathrm{d}}{\mathrm{d }t}\left(\frac{1}{2} \int_{\Omega}\big|\nabla u_k\big|^2 \mathrm{d}x\right) \\
&\leq C I + C II
\leq C\int_{\Omega}\big|\nabla u_k\big|^2 \mathrm{d}x + o(\varepsilon).
\end{align*}
This immediately implies
\[
\frac{\mathrm{d}}{\mathrm{d} t}\left(\frac{1}{2} \int_{\Omega}\big|\nabla u_k\big|^2 \mathrm{d}x\right)
\leq C\int_{\Omega}\big|\nabla u_k\big|^2 \mathrm{d}x + o(\varepsilon).
\]
Applying the Gronwall inequality, we arrive at
\begin{equation}\label{ebi6}
\int_{\Omega}\big|\nabla u_k\big|^2 \mathrm{d}x
\leq Ce^T \int_{\Omega}\big|\nabla u_{0,k}\big|^2 \mathrm{d}x
\leq Ce^T \int_{\Omega}\big|\nabla u_{0}\big|^2 \mathrm{d}x.
\end{equation}
Using \eqref{ebi2}, \eqref{ebi6} and the Sobolev inequality again, we further get
\begin{align}\label{ebi3}
\int_0^t\int_{\Omega}\varphi_k(u_{k})u_{k,t}^2\mathrm{d}x\mathrm{d}\tau
&\leq \big|E(u_k(0))\big|+\big|E(u_k(t))\big| \nonumber\\
&\leq \int_{\Omega}|\nabla u_{0,k}|^2\mathrm{d}x+\int_{\Omega}|\nabla u_{k}|^2\mathrm{d}x
+\int_{\Omega}u_{0,k}^{p+1}\mathrm{d}x+\int_{\Omega}u_{k}^{p+1}\mathrm{d}x \nonumber\\
&\leq C \int_{\Omega}|\nabla u_{0}|^2\mathrm{d}x.
\end{align}
Therefore, there exist a function $u$ and a subsequence of $\{u_k\}$ such that
\[
\begin{aligned}
& u_k \rightharpoonup u \quad \text{\ weakly\  star\ in\  } L^{\infty}\big(0, T ; H_0^1(\Omega)\big), \\
& u_k \rightharpoonup u \quad \text{\ weakly\  star\ in\  } L^{\infty}\big(0, T ; L^{p+1}(\Omega)\big), \\
& u_k \to u \quad \text{a.e. in } Q=\Omega \times[0, T), \\
&\frac{\partial}{\partial t}\Big(u_k^{\frac{p+1}{2}}  \Big)
\rightharpoonup \frac{\partial}{\partial t}\Big(u^{\frac{p+1}{2}} \Big)
\quad \text{\ weakly\  star\ in\   } L^2\big([0, T); L^2(\Omega)\big).
\end{aligned}
\]
These convergence results yield
\[
\big((u^p)_t, w\big)+\big(\nabla u, \nabla w\big) =\big(|u|^{p},w\big)
\]
for all $w \in X$ and $t\in[0,T)$.
\end{proof}

\begin{proof}[Proof of Theorem \ref{t22}]
For any fixed $t_0 \in\left(0, \frac{T}{4}\right)$, there exists a cutoff function
 $\eta \in C_c^{\infty}\left(\frac{t_0}{2}, T\right)$ satisfying $0 \leq \eta \leq 1$ on $(0, T)$
  and $\left| \eta^{\prime}\right|<\frac{1}{t_0} \eta$.
  Such a function $\eta$ can be explicitly constructed via standard smooth bump functions
  (the AI Chatgpt or Deepseek  can easily find such a function \(\eta\)).

Let $\psi(x, t)=u^{2 s+p} \eta^2$. We apply $\psi$ as a test function to Definition \ref{wd} so as to obtain
\begin{equation}\label{ff01}
 \int_0^T \int_{\Omega}\left[\psi pu^{p-1}u_t+\nabla u \nabla \psi-\psi u^{p} \right] \mathrm{d} x \mathrm{~d} t=0,
\end{equation}
for $s \geq 0$ to be chosen later. Suppose $u \in L^{2 s+2}\left(Q_T\right)$.
Performing standard manipulations, we obtain the estimate of the first term on the right-hand side of Equation (\ref{ff01}),
\begin{align} \label{ff02}
\int_0^T \int_{\Omega} \psi pu^{p-1} u_t \mathrm{~d} x \mathrm{~d} t&=  \int_0^T \int_{\Omega}p u_t u^{2 s+2p-1} \eta^2 \mathrm{~d} x \mathrm{~d} t \nonumber\\
&=  \frac{p}{2s+2p} \int_0^T \int_{\Omega}\left(u^{2 s+2p}\right)_t \eta^2 \mathrm{~d} x \mathrm{~d} t \nonumber\\
&=  \frac{p}{2s+2p} \int_0^T \int_{\Omega}\left(u^{2 s+2p} \eta^2\right)_t \mathrm{~d} x \mathrm{~d} t \nonumber\\
& \quad-\frac{2p}{2s+2p} \int_0^T \int_{\Omega} u^{2 s+2p} \eta \eta_t \mathrm{~d} x \mathrm{d} t\nonumber\\
&= -\frac{2p}{2s+2p} \int_0^T \int_{\Omega} u^{2 s+2p} \eta \eta_t \mathrm{~d} x \mathrm{d} t\nonumber\\
&\geq-\frac{2p}{2s+2p}\frac{1}{t_0} \int_0^T \int_{\Omega} u^{2 s+2p} \eta^2 \mathrm{~d} x \mathrm{d} t.
\end{align}
The second term on the left-hand side of Equation (\ref{ff01}) can be estimated as
follows:
\begin{align}\label{ff03}
\int_0^T \int_{\Omega} \nabla u \nabla\left(u^{2 s+p} \eta^2\right) \mathrm{d} x
\mathrm{~d} t & =\frac{2 s+p}{\left(s+\frac{1}{2}+\frac{p}{2}\right)^2} \int_0^T \int_{\Omega}\left|\nabla u^{s+\frac{p+1}{2}}\right|^2 \eta^2 \mathrm{~d} x \mathrm{~d} t \nonumber \\
& \geq\frac{1}{s+p} \int_0^T \int_{\Omega}\left|\nabla u^{s+\frac{1}{2}+\frac{p}{2}}\right|^2 \eta^2 \mathrm{~d} x \mathrm{~d} t.
\end{align}
By H\"{o}lder  inequality we can estimate the right-hand term of \eqref{ff02} and  the third term on the left-hand side
of Equation (\ref{ff01}) as follows
\begin{align}\label{ff04}
& \int_0^T \int_{\Omega} \psi u^{p} \mathrm{~d} x \mathrm{~d} t \nonumber\\
& =\int_0^T \int_{\Omega} u^{p-1} u^{2 s+p+1} \eta^2 \mathrm{~d} x \mathrm{~d} t \nonumber\\
& \leq M \int_0^T \int_{\Omega} u^{2 s+p+1} \eta^2 \mathrm{~d} x \mathrm{~d} t+\int_0^T \eta^2\left(\int_{\left\{u^{p-1} \geq M\right\}} u^{\frac{4}{N-2} \frac{N}{2}} \mathrm{~d} x\right)^{\frac{2}{N}} \nonumber\\
& \quad \times\left(\int_{\Omega}\left(u^{2 s+p+1}\right)^{\frac{N}{N-2}} \mathrm{~d} x\right)^{\frac{N-2}{N}} \mathrm{~d} t \nonumber\\
& \leq M \int_0^T \int_{\Omega} u^{2 s+p+1} \eta^2 \mathrm{~d} x \mathrm{~d} t+\frac{1}{S} \sup _t\left(\int_{\left\{u^{p-1} \geq M\right\}} u^{p+1} \mathrm{~d} x\right)^{\frac{2}{N}} \nonumber\\
& \quad \times \int_0^T \eta^2 \int_{\Omega}\left|\nabla u^{s+\frac{1}{2}+\frac{p}{2}}\right|^2 \mathrm{~d} x \mathrm{~d} t
\end{align}
where S is the best Sobolev constant. From Equations \eqref{ff01}-\eqref{ff04}, we have
\begin{align}
\frac{1}{s+p} \int_0^T \int_{\Omega}\left|\nabla u^{s+\frac{1}{2}+\frac{p}{2}}\right|^2 \eta^2 \mathrm{~d} x \mathrm{~d} t
  &\leq 
  C \int_0^T \int_{\Omega} u^{2 s+p+1} \eta^2 \mathrm{~d} x \mathrm{~d} t \nonumber\\
 &+\frac{1}{S} \sup _t\left(\int_{\left\{u^{p-1} \geq M\right\}} u^{p+1} \mathrm{~d} x\right)^{\frac{2}{N}} \int_0^T \int_{\Omega}\left|\nabla u^{s+\frac{1}{2}+\frac{p}{2}}\right|^2 \eta^2 \mathrm{~d} x \mathrm{~d} t .
\end{align}
Now we have the following claim:

\textbf{Claim.} Let

$$
\varepsilon(M)=\frac{1}{S} \sup _t\left(\int_{\left\{u^{p-1} \geq M\right\}} u^{p+1} \mathrm{~d} x\right)^{\frac{2}{N}},
$$
then $\varepsilon(M) \rightarrow 0$ as $M \rightarrow \infty$.

\textbf{Proof of the Claim:} Let
$$g(t)=\int_{\Omega} u^{p+1} \mathrm{~d} x,$$
 and
 $$g_M(t)=\int_{\left\{u^{p-1}\geq M\right\}} u^{p+1} \mathrm{~d} x.$$
We have to prove that for any $\epsilon>0$, there exists a positive number $M_0 \in \mathbb{R}^{+}$ such that
$$
g_M(t)<\epsilon, \quad \text { as } M>M_0, t \in(0, \infty) .
$$
From \eqref{e0538}, we can choose a $T$ such that
$$
g_M(t) \leq g(t)<\epsilon, \quad \text { for } t>T
$$
So we will complete the proof if we prove that
$$
g_M(t)<\epsilon, \quad \text { as } M>M_0, \text { for } t \in[0, T] .
$$
Indeed, by means of $u \in C\left([0, T]; X\right)$, we obtain $g(t)$ is continuous function on $[0, T]$. So it is uniformly continuous function on $[0, T]$. The rest is obvious, which completes the proof of the claim.

Now we can complete the proof of $L^q$-estimate. Let
$$s_0=0, 2s_0+p+1=p+1, 2s_i+p+1=\left(2s_{i-1}+p+1\right.)\left(1+\frac{2}{N}\right), $$
for $i \geq 1$. For any given $q (1\leq q<\infty)$,
there exists a $i_0$ such that
$$ \left(2s_{i_0}+p+1\right)<q \leq \left(2s_{i_0+1}+p+1\right)=\left(2s_{i_0}+p+1\right)\left(1+\frac{2}{N}\right), $$
and choosing $M$ such that

$$
\varepsilon(M)=\frac{1}{q}<\frac{1}{2\left(s_{i_0-1}+1\right)} .
$$
We may now conclude that
\begin{align}
& \frac{1}{s+p} \int_0^T \int_{\Omega}\left|\nabla u^{s+\frac{1}{2}+\frac{p}{2}}\right|^2 \eta^2 \mathrm{~d} x \mathrm{~d} t \nonumber\\
& \quad \leq 
C \int_0^T \int_{\Omega} u^{2 s+p+1} \eta^2 \mathrm{~d} x \mathrm{~d} t
\end{align}
By Sobolev inequality, we get
\begin{align}
\int_0^T & \int_{\Omega} u^{(2 s+p+1)\left(1+\frac{2}{N}\right)} \eta^{2\left(1+\frac{2}{N}\right)} \mathrm{d} x \mathrm{~d} t \nonumber\\
& \leq C \frac{1}{S}\left(\int_0^T \int_{\Omega} u^{2 s+p+1} \eta^2 \mathrm{~d} x\right)^{\frac{2}{N}}\left(\int_0^T \int_{\Omega}\left|\nabla u^{s+\frac{p+1}{2}}\right|^2 \eta^2 \mathrm{~d} x \mathrm{~d} t\right) \nonumber\\
& \leq C\left(\int_0^T \int_{\Omega} u^{2 s+p+1} \eta^2 \mathrm{~d} x \mathrm{~d} t\right)^{1+\frac{2}{N}},
\end{align}
where $C$ depends on $N, s, p$.

Therefore, from \eqref{e0538}, $u \in W_q^{2,1}\left(Q_{\infty}\right)$ for any $1 \leq q<\infty$,
where $Q_{\infty}=\Omega \times\left[t_0, \infty\right)$, and
we apply the standard bootstrap argument \cite{rlady}
to obtain $u \in C^{(2, \alpha)(1, \alpha)}\left(Q_{\infty}\right)$.
Thus, $u(x, t)$ is a classical solution for all $t \geq t_0>0$, which completes the proof of Theorem \ref{t22}.
\end{proof}

\section{Low initial energy $E\left(u_0\right)<d$}
In this section, we consider solutions with low initial energy $E\left(u_0\right)<d$.
We discuss finite-time extinction and infinite-time blowup phenomena for the problem \eqref{e01}.

\begin{theorem}\label{t31}
   Assume that $u_0 \in X$, $T$ is the maximal existence time of $u$, and $0<e<d, \delta_1<\delta_2$ are two roots of equation $d(\delta)=e$. We have
 \begin{description}
   \item[(1)]  If $D\left(u_0\right)>0$, all weak solutions $u$ of problem (\ref{e01}) with $E\left(u_0\right)=e$ belong to $W_\delta$ for $\delta_1<\delta<\delta_2, 0 \leqslant t<T$;

   \item[(2)] If $D\left(u_0\right)<0$, all weak solutions $u$ of problem (\ref{e01}) with $E\left(u_0\right)=e$ belong to $V_\delta$ for $\delta_1< \delta<\delta_2, 0 \leqslant t<T$.
 \end{description}
\end{theorem}

\begin{theorem}\label{t32}
 (Global existence). Assume that $u_0 \in X, E\left(u_0\right)<d, D\left(u_0\right)>0$. Then problem (\ref{e01})  has a global solution $u(t) \in L^{\infty}\left(0, \infty ; X\right)$ and $u(t) \in W$ for $0 \leqslant t<\infty$.
\end{theorem}

\begin{theorem}\label{t33}
Assume that $u_0 \in X, E\left(u_0\right)<d$ and $D\left(u_0\right)<0$. Then the weak solution $u(t)$ of problem (\ref{e01})  blows up in infinite time, that is, there exists a $T>0$ such that
$$
\lim _{t \rightarrow +\infty} \int_0^t\|u(\tau)\|_{p+1}^{p+1} \mathrm{d} \tau=+\infty.
$$
\end{theorem}

\begin{theorem}\label{t34}
 Assume that $u_0 \in X, E\left(u_0\right)<d$ and $D\left(u_0\right)>0, \delta_1<\delta_2$ are the two roots of equation $d(\delta)=E\left(u_0\right)$.
  Then, for the weak solution $u$ of problem (\ref{e01}) undergoes finite-time extinction, that is,

\begin{equation}\label{e0538}
 \|u\|_{p+1}^{p+1}  \leq
\begin{cases}
\left(B_0^{1-\gamma}-C(1-\gamma) t\right)^{\frac{1}{1-\gamma}}, & 0 \leq t<t^{*}, \\
0, & t \geq t^{*},
\end{cases}
\end{equation}
where $\gamma=\frac{2}{p+1}$,  $B_0= \left\|u(0)\right\|_{p+1}^{p+1},
C=S\left(1-\delta_1\right)\frac{(p+1)}{p}$,
$t^{*} = \displaystyle\frac{B_0^{1-\gamma}}{C(1-\gamma)}$,  $S$ is given in (\ref{e03}).
\end{theorem}

In order to prove Theorems \ref{t31}-\ref{t34}, we need the following lemmas:
\begin{lemma}\label{l31}
 (Energy identity) For $0<T \leq \infty$, assume that $u: \Omega \times[0, T) \rightarrow \mathbb{R}$ is a weak solution to problem (\ref{e01}). Then it holds
\begin{align}\label{e08}
 \int_{t_1}^{t_2} \int_{\Omega} pu^{p-1}u^2_{\tau}\mathrm{d}x\mathrm{d}\tau+E(u(t_2))=E(u(t_1)), \quad \forall t_1, t_2 \in (0, T).
\end{align}
\end{lemma}
\begin{proof}
 Multiplying (\ref{e01}) by $u_t$ and integrating over $\Omega$ via the integration by parts, we get (\ref{e08}).
\end{proof}

\begin{lemma}\label{l32}
 If $0<E(u)<d$ for some $u \in X$, and $\delta_1<1<\delta_2$ are the two roots of equation $d(\delta)=E(u)$, then the sign of $D_{\delta}(u)$ does not change for $\delta_1<\delta<\delta_2$.
\end{lemma}

\begin{proof}
Since $E(u)>0$, we have $\|\nabla u\|_{2} \neq 0$. If the sign of $D_{\delta}(u)$ is changeable for $\delta_1<\delta<\delta_2$, then we choose $\bar{\delta} \in\left(\delta_1, \delta_2\right)$ such that $D_{\bar{\delta}}(u)=0$. Hence, by the definition of $d(\bar{\delta})$, we can obtain $E(u) \geqslant d(\bar{\delta})$, which contradicts $E(u)=d\left(\delta_1\right)=d\left(\delta_2\right)<d(\bar{\delta})$.
\end{proof}

\begin{proof}[Proof of Theorem \ref{t31}]
  (1) Let $u(t)$ be any weak solution of problem (\ref{e01}) with $E\left(u_0\right)=e$, $D\left(u_0\right)>0$, and $T$ be the maximal existence time of $u(t)$. Using $E\left(u_0\right)=e, D\left(u_0\right)>0$ and Lemma \ref{l32},
  we have $D_{\delta}\left(u_0\right)>0$ and $E\left(u_0\right)<d(\delta)$. So $u_0(x) \in W_\delta$ for $\delta_1<\delta<\delta_2$. We need to prove that $u(t) \in W_\delta$ for $\delta_1<\delta<\delta_2$ and $0<t<T$. Indeed, if this is not the conclusion, from time continuity of $D(u)$ we assume that there must exist a $\delta_0 \in\left(\delta_1, \delta_2\right)$ and $t_0 \in(0, T)$ such that $u\left(t_0\right) \in \partial W_{\delta_0}$,
  and $D_{\delta_0}\left(u\left(t_0\right)\right)=0,\left\|\nabla u\left(t_0\right)\right\|_{2} \neq 0$ or $E\left(u\left(t_0\right)\right)=d\left(\delta_0\right)$. From the energy equality
\begin{equation}\label{e09}
\int_{0}^{t}\int_{\Omega} pu^{p-1}u^2_{\tau}\mathrm{d}x\mathrm{d}\tau+E(u(t))=E\left(u_0\right)<d(\delta), \delta_1<\delta<\delta_2, \quad 0 \leqslant t<T,
\end{equation}
we easily know that $E\left(u\left(t_0\right)\right) \neq d\left(\delta_0\right)$. If $D_{\delta_0}\left(u\left(t_0\right)\right)=0,
\left\|\nabla u\left(t_0\right)\right\|_{2} \neq 0$, then by the definition of $d(\delta)$ we obtain $E\left(u\left(t_0\right)\right) \geqslant d\left(\delta_0\right)$, which contradicts (\ref{e09}).

(2) Let $u(t)$ be any weak solution of problem (\ref{e01}) with $E\left(u_0\right)=e, D\left(u_0\right)<0$,
and $T$ be the maximal existence time of $u(t)$. Using $E\left(u_0\right)=e, D\left(u_0\right)<0$ and
Lemma \ref{l32}, we have $D_{\delta}\left(u_0\right)<0$ and $E\left(u_0\right)<d(\delta)$. So $u_0 \in V_\delta$
for $\delta_1<\delta<\delta_2$. We need to prove that $u(t) \in V_\delta$ for $\delta_1<\delta<\delta_2$ and $0<t<T$.
 Indeed, if this is not the conclusion, from time continuity of $D_{\delta}(u)$ we assume that there must exist a
 $\delta_0 \in\left(\delta_1, \delta_2\right)$ and $t_0 \in(0, T)$
 such that $u\left(t_0\right) \in \partial V_{\delta_0}$, and $D_{\delta_0}\left(u\left(t_0\right)\right)=0$,
 or $E\left(u\left(t_0\right)\right)=d\left(\delta_0\right)$. From the energy equality (\ref{e09}),
 we easily know that $E\left(u\left(t_0\right)\right) \neq d\left(\delta_0\right)$. If $D_{\delta_0}\left(u\left(t_0\right)\right)=0$,
 and $t_0$ is the first time such that $D_{\delta_0}(u(t))=0$, then $D_{ \delta_0}(u(t))<0$ for $0 \leqslant t<T$. By Lemma (\ref{l03}) (2),
  we have $\sqrt{A(u\left(t_0\right))}>r\left(\delta_0\right)$ for $0 \leqslant t<T$.
  So, $\sqrt{A(u\left(t_0\right))}>r\left(\delta_0\right)$ and $E\left(u\left(t_0\right)\right) \geq  d\left(\delta_0\right)$, which contradicts (\ref{e09}). As required.
\end{proof}

\begin{proof}[Proof of Theorem \ref{t32}]
From Theorem \ref{t21}, the local existence result for problem \eqref{e01}
holds in a more general setting
where the initial datum satisfies $u_0 \in X$,
and the corresponding solution obeys $u \in L^{\infty}([0, T] ; X)$
together with $\partial_t\left(u^{\frac{p+1}{2}}\right) \in L^2\left(0, T ; L^2(\Omega)\right)$.

Using $E\left(u_0\right)<d, D\left(u_0\right)>0$ and Lemma \ref{l32}, we have $D_{\delta}\left(u_0\right)>0$ and $E\left(u_0\right)<d(\delta)$.
 So $u_0 \in W_\delta$ for $\delta_1<\delta<\delta_2$. We need to prove that
 $u(t) \in W_\delta$ for $\delta_1<\delta<\delta_2$ and $0<t<T$.
 Indeed, if this is not the conclusion, from time continuity
 of $D_{\delta}(u)$ we assume that there must exist a $\delta_0 \in\left(\delta_1, \delta_2\right)$
 and $t_0 \in(0, T)$ such that $u\left(t_0\right) \in \partial W_{\delta_0}$, and
 $D_{\delta_0}\left(u\left(t_0\right)\right)=0,\left\|\nabla u(t_0)\right\|_{2} \neq 0$
 or $E\left(u\left(t_0\right)\right)=d\left(\delta_0\right)$. From the energy equality
\begin{equation}\label{e10}
\int_{0}^{t}\int_{\Omega} pu^{p-1}u^2_{\tau}\mathrm{d}x\mathrm{d}\tau+E(u(t))=E\left(u_0\right)<d(\delta), \delta_1<\delta<\delta_2, \quad 0 \leqslant s<S,
\end{equation}
we easily know that $E\left(u\left(t_0\right)\right) \neq d\left(\delta_0\right)$.
If $D_{ \delta_0}\left(u\left(t_0\right)\right)=0,\left\|\nabla u\left(t_0\right)\right\|_{2} \neq 0$,  then by the definition of $d(\delta)$ we obtain $E\left(u\left(t_0\right)\right) \geqslant d\left(\delta_0\right)$, which contradicts (\ref{e10}).
\end{proof}

\begin{proof}[Proof of Theorem \ref{t33}]
 We argue by contradiction. Suppose that there exists a global weak solution $u(t)$. Set
\begin{equation*}
f(t)=\int_0^t\|u\|_{p+1}^{p+1} \mathrm{~d} \tau, t>0.
\end{equation*}
Multiplying (\ref{e01}) by $u$ and integrating over $\Omega \times(0, t)$, we get
\begin{align*}
\int_0^t\int_{\Omega}pu^{p}u_{\tau}\mathrm{d}x\mathrm{d}\tau&=\frac{p}{p+1}\left(\|u(t)\|_{p+1}^{p+1}-\|u(0)\|_{p+1}^{p+1}\right)\nonumber\\
&=\int_0^{t}\int_{\Omega}\left( u^{p+1}-|\nabla u|^2\right)\mathrm{d}x\mathrm{d}\tau\nonumber\\
&=\int_0^{t}(B(u)-A(u))d\tau.
\end{align*}
According to the definition of $f(t)$, we have $f^{\prime}(t)=\|u(t)\|_{p+1}^{p+1}$ and hence
\begin{equation}\label{e105}
\begin{aligned}
f^{\prime}(t)&=\|u\|_{p+1}^{p+1}=\|u(0)\|_{p+1}^{p+1}+\frac{p+1}{p}\int_0^t\int_{\Omega}\left( u^{p+1}-|\nabla u|^2\right)dxd\tau\nonumber \\
   &=  (p+1)\int_0^t\int_{\Omega}  u^p u_{\tau} \mathrm{d} x\mathrm{d}\tau+\|u(0)\|_{p+1}^{p+1}
\end{aligned}
\end{equation}
and

\begin{align}\label{e11}
  f^{\prime\prime}(t)&=\frac{p+1}{p}\int_{\Omega}\left( u^{p+1}-|\nabla u|^2\right)dx\nonumber\\
  &=\frac{p+1}{p}(B(u)-A(u))=-\frac{p+1}{p}D(u).
\end{align}
Then by \eqref{e105}, we deduce
\begin{align}\label{e106}
  (f^{\prime}(t))^2&=  \left[(p+1)\int_0^t\int_{\Omega}  u^p u_{\tau} \mathrm{d} x\mathrm{d}\tau+\|u(0)\|_{p+1}^{p+1}\right]^2\nonumber\\
  &=\left[(p+1) \int_0^t \int_{\Omega} u^p u_\tau \mathrm{d} x \mathrm{d} \tau\right]^2\nonumber\\
  &\quad +2(p+1)\|u(0)\|_{p+1}^{p+1} \int_0^t \int_{\Omega} u^p u_\tau \mathrm{d} x \mathrm{d} \tau+\|u(0)\|_{p+1}^{2p+2}\nonumber\\
  &=\left[(p+1) \int_0^t \int_{\Omega} u^p u_\tau \mathrm{d} x \mathrm{d} \tau\right]^2\nonumber\\
  &\quad +2(p+1)\|u(0)\|_{p+1}^{p+1} \int_0^t \int_{\Omega} u^p u_\tau \mathrm{d} x \mathrm{d} \tau+2\|u(0)\|_{p+1}^{2p+2}-\|u(0)\|_{p+1}^{2p+2}\nonumber\\
  &=\left[(p+1) \int_0^t \int_{\Omega} u^p u_\tau \mathrm{d} x \mathrm{d} \tau\right]^2
  +2\|u(0)\|_{p+1}^{p+1} f^{\prime}(t)-\|u(0)\|_{p+1}^{2p+2}.
\end{align}
Now using (\ref{e10}), (\ref{e11}) and
$$
\begin{aligned}
E(u) & =\frac{1}{2}A(u)-\frac{1}{p+1}B(u) \\
& =\frac{p-1}{2(p+1)}A(u)+\frac{1}{p+1} D(u),
\end{aligned}
$$
we can obtain
\begin{align}\label{e1065}
  f^{\prime\prime}(t)&=\frac{p+1}{p}\int_{\Omega}\left( u^{p+1}-|\nabla u|^2\right)\mathrm{d}x=-\frac{p+1}{p} D(u)\nonumber\\
  &=\frac{p^2-1}{2p}A(u)-\frac{(p+1)^2}{p}E(u)\nonumber\\
  &= \frac{p^2-1}{2p}A(u)+(p+1)^2\int_{0}^{t} \int_{\Omega}  u^{p-1} u_\tau^2 \mathrm{d} x \mathrm{d} \tau-\frac{(p+1)^2}{p}E\left(u\left(0\right)\right).
\end{align}
Continuing, in view of (\ref{e03}), we have
\begin{align}
 f(t)f^{\prime\prime}(t)&=  f(t)\left[\frac{p^2-1}{2p}A(u)+(p+1)^2\int_{0}^{t} \int_{\Omega}  u^{p-1} u_\tau^2 \mathrm{d} x \mathrm{d} \tau-\frac{(p+1)^2}{p}E\left(u\left(0\right)\right)\right]\nonumber\\
 &= f(t)\frac{p^2-1}{2p}A(u)
 +(p+1)^2 \int_0^t\|u\|_{p+1}^{p+1} \mathrm{~d} \tau \int_{0}^{t} \int_{\Omega}  u^{p-1} u_\tau^2 \mathrm{d} x \mathrm{d} \tau\nonumber\\
 &\quad-\frac{(p+1)^2}{p} \int_0^t\|u\|_{p+1}^{p+1} \mathrm{~d} \tau E\left(u\left(0\right)\right)\nonumber\\
 &\geq  f(t)\frac{p^2-1}{2p}S(f^{\prime}(t))^{\frac{2}{p+1}}
 +(p+1)^2 \int_0^t\|u\|_{p+1}^{p+1} \mathrm{~d} \tau \int_{0}^{t} \int_{\Omega}  u^{p-1} u_\tau^2 \mathrm{d} x \mathrm{d} \tau\nonumber\\
 &\quad-\frac{(p+1)^2}{p} \int_0^t\|u\|_{p+1}^{p+1} \mathrm{~d} \tau E\left(u\left(0\right)\right).
\end{align}
Hence, from \eqref{e106} we have
\begin{equation}\label{e108}
\begin{aligned}
 f(t) f^{\prime \prime}(t)-\left(f^{\prime}(t)\right)^2 &
   \geq (p+1)^2 \int_0^t\|u\|_{p+1}^{p+1} \mathrm{~d} \tau \int_0^t \int_{\Omega} u^{p-1} u_\tau^2 \mathrm{d} x \mathrm{d} \tau\nonumber
   \\ &\quad - \left[(p+1) \int_0^t \int_{\Omega} u^p u_\tau \mathrm{d} x \mathrm{d} \tau\right]^2 \nonumber\\
   &\quad+ \frac{p^2-1}{2 p} S\left(f^{\prime}(t)\right)^{\frac{2}{p+1}} f(t)-\frac{(p+1)^2}{p} \int_0^t\|u\|_{p+1}^{p+1} \mathrm{~d} \tau E(u(0))\nonumber\\
   &\quad-2\|u(0)\|_{p+1}^{p+1} f^{\prime}(t).
\end{aligned}
\end{equation}
Making use of the Schwartz inequality, we have
\begin{align}\label{e107}
 \left[\int_0^t \int_{\Omega}u^pu_{\tau}\mathrm{d}\tau\right]^2&=\left[\int_0^t \int_{\Omega}u^{\frac{p-1}{2}} u^{\frac{p+1}{2}} u_{\tau}\mathrm{d}\tau\right]^2\\
  &\leq \int_0^t\int_{\Omega} u^{p+1}\mathrm{d}x \mathrm{~d} \tau \int_0^t \int_{\Omega} u^{p-1} u_\tau^2 \mathrm{d} x \mathrm{d} \tau.
\end{align}
By combining (\ref{e108}) and (\ref{e107}), we obtain that
$$
\begin{aligned}
f(t) f^{\prime \prime}(t)-\left(f^{\prime}(t)\right)^2
\geq & \frac{p^2-1}{2 p} S\left(f^{\prime}(t)\right)^{\frac{2}{p+1}} f(t)-\frac{(p+1)^2}{p} \int_0^t\|u\|_{p+1}^{p+1} \mathrm{~d} \tau E(u(0))\nonumber\\
   &\quad-2\|u(0)\|_{p+1}^{p+1} f^{\prime}(t).
\end{aligned}
$$
Next, we distinguish two cases:

(1) If $E\left(u_0\right) \leqslant 0$, then
\begin{align*}
f(t) f^{\prime \prime}(t)-\left(f^{\prime}(t)\right)^2 &\geq \frac{p^2-1}{2 p} S\left(f^{\prime}(t)\right)^{\frac{2}{p+1}} f(t)
-2\|u(0)\|_{p+1}^{p+1} f^{\prime}(t)\\
&= \frac{p^2-1}{2 p} S\left(f^{\prime}(t)\right)^{\frac{2}{p+1}} \left[f(t)
-2\|u(0)\|_{p+1}^{p+1} (f^{\prime}(t))^{1-\frac{2}{p+1}}\right].
\end{align*}
Now we prove
\begin{equation}\label{sue}
D(u)<0 \ \text{for}\  t>0.
\end{equation}
If not, we must be allowed to choose a $t_0>0$
such that $D\left(u\left(t_0\right)\right)=0$ and $D(u)<0$ for $0 \leqslant t<t_0$.
From Lemma \ref{l03} (2), we have  $\sqrt{A(u)}>r(1)$ for $0 \leqslant t<t_0, \sqrt{A(u(t_0))} \geqslant r(1)$
and $E\left(u\left(t_0\right)\right) \geqslant d$, which contradicts (\ref{e08}).

From \eqref{e11}, we know that $f^{\prime\prime}(t)>0$, $f^{\prime}(t)>0$ and  $f(t)>0$, for $t >0$. So for $t \geqslant t_0$, by Lemma \ref{l4455},  we have
$$
f(t) f^{\prime \prime}(t)-\left(f^{\prime}(t)\right)^2>0.
$$

(2) If $0<E\left(u_0\right)<d$, then by Theorem \ref{t31} we have $u(t) \in V_\delta$ for $1<\delta<\delta_2, t \geqslant 0$,
and $D_\delta(t)<0,\sqrt{A(u)}>r(\delta)$ for $1<\delta<\delta_{2}, t \geqslant 0$,
where $\delta_2$ is the larger root of equation
 $d(\delta)=E\left(u_0\right)$. Hence, $D_{\delta_2}(u) \leqslant 0$
 and $\sqrt{A(u)} \geqslant r\left(\delta_2\right)$ for $t \geqslant 0$.
 By (\ref{e11}), we have

$$
\begin{aligned}
f^{\prime \prime}(t) & =-2 D(u)=2\left(\delta_2-1\right)\left(A(u)\right)-2 D_{\delta_2}(u), \\
& \geqslant 2\left(\delta_2-1\right)\left(A(u)\right) \geqslant 2\left(\delta_2-1\right) r^2\left(\delta_2\right), \quad t \geqslant 0,  \\
f^{\prime}(t) & \geqslant 2\left(\delta_2-1\right) r^2\left(\delta_2\right) t+f^{\prime}(0) \geqslant 2\left(\delta_2-1\right) r^2\left(\delta_2\right) t, \quad t \geqslant 0, \\
f(t) & \geqslant\left(\delta_2-1\right) r^2\left(\delta_2\right) t^2, \quad t \geqslant 0 .
\end{aligned}
$$
Therefore, for sufficiently large $t$,
 the above inequalities and  Lemma \ref{l4455}  imply  that
$$
\begin{aligned}
f(t) f^{\prime \prime}(t)-\left(f^{\prime}(t)\right)^2
\geq & \frac{p^2-1}{2 p} S\left(f^{\prime}(t)\right)^{\frac{2}{p+1}} f(t)-\frac{(p+1)^2}{p} \int_0^t\|u\|_{p+1}^{p+1} \mathrm{~d} \tau E(u(0))\nonumber\\
   &\quad-2\|u(0)\|_{p+1}^{p+1} f^{\prime}(t)\\
   \geq & \frac{p^2-1}{4 p} S\left(f^{\prime}(t)\right)^{\frac{2}{p+1}} f(t)-\frac{(p+1)^2}{p} \int_0^t\|u\|_{p+1}^{p+1} \mathrm{~d} \tau E(u(0))\nonumber\\
   &\quad \frac{p^2-1}{4 p} S\left(f^{\prime}(t)\right)^{\frac{2}{p+1}} f(t)-2\|u(0)\|_{p+1}^{p+1} f^{\prime}(t)>0.
\end{aligned}
$$
From the above calculations, we have obtained
\begin{equation*}
f^{\prime \prime}(t) f(t)-\left[f^{\prime}(t)\right]^2>0.
\end{equation*}
Then a direct computation yields
\begin{equation}\label{ef41}
(\log |f(t)|)^{\prime}=\frac{f^{\prime}(t)}{f(t)}
\end{equation}
and
\begin{equation}\label{ef42}
(\log |f(t)|)^{\prime \prime}=\left(\frac{f^{\prime}(t)}{f(t)}\right)^{\prime}=\frac{f^{\prime \prime}(t) f(t)-\left[f^{\prime}(t)\right]^2}{f^2(t)}>0 .
\end{equation}It follows from \eqref{ef42} that \((\log |f(t)|)^{\prime}=\dfrac{f^{\prime}(t)}{f(t)}\)
 is strictly increasing with respect to $t$. Integrating \eqref{ef41} from \(t_0\) to $t$, we get
\begin{align*}
\log |f(t)|-\log \left|f\left(t_0\right)\right|&=\int_{t_0}^t\big(\log |f(\tau)|\big)^{\prime} \mathrm{d} \tau\\
&=\int_{t_0}^t \frac{f^{\prime}(\tau)}{f(\tau)} \mathrm{d} \tau \geq \frac{f^{\prime}\left(t_0\right)}{f\left(t_0\right)}\left(t-t_0\right),
\end{align*}
where \(0 \leq t_0<t\). Consequently,
\begin{equation}\label{ef43}
f(t) \geq f\left(t_0\right) \exp \left(\frac{f^{\prime}\left(t_0\right)}{f\left(t_0\right)}\left(t-t_0\right)\right).
\end{equation}Using inequality \eqref{ef43}, we conclude \(\lim _{t \rightarrow+\infty} f(t)=+\infty\).
This implies that the solution to problem \eqref{e01} blows up as \(t \to +\infty\). The proof is complete.

\end{proof}

\begin{lemma}\label{l4455}
Suppose \(t\geq0\), \(0<\alpha<1\), \(f''(t)>0\), \(f(0)=0\) and \(f'(0)\geq 0\). The following inequality holds:
$$f(t)-\big[f'(t)\big]^{\alpha}>0$$
for sufficiently large t.
\end{lemma}
\begin{proof}
It follows from \(f''(t)>0\) that \(f'(t)\) is strictly increasing on \([0,+\infty)\). Combining with \(f(0)=0\), we have
$$f(t)=\int_0^t f'(s) \mathrm{d} s\geq \int_{t_0}^t f'(s) \mathrm{d} s\geq f^{\prime}(t_0)(t-t_0).$$
Hence, we get  $$\lim\limits_{t \to +\infty} f(t) = +\infty.$$

 If $f^{\prime}(t)$ is bounded above, the conclusion is trivial. So, we assume that

 $$
\lim _{t \rightarrow+\infty} f^{\prime}(t)=+\infty.
$$
By contradiction, suppose that for sufficiently large $t$,
$$
\frac{f(t)}{\left[f^{\prime}(t)\right]^\alpha}\leq 1,
$$
that is,

$$
\frac{f^{\prime}(t)}{\left[f(t)\right]^\frac{1}{\alpha}}\geq 1.
$$
Then, one has
$$
\int_{t_0}^t\frac{f^{\prime}(s)}{\left[f(s)\right]^\frac{1}{\alpha}}\mathrm{d}s\geq t-t_0,
$$
So,  $$\frac{\alpha}{\alpha-1}\left(f(t)\right)^{1-\frac{1}{\alpha}}\geq \frac{\alpha}{\alpha-1}\left(f(t_0)\right)^{1-\frac{1}{\alpha}}+(t-t_0).$$
Take the limit of both sides of the above formula as $t\to+\infty$. This leads to a contradiction. The proof is complete.
\end{proof}

\begin{proof}[Proof of Theorem \ref{t34}]
Multiplying (\ref{e01}) by $w, w \in L^{\infty}\left(0, \infty ; X\right)$, we have
\begin{equation}\label{e12}
\left(pu^{p-1}u_t, w\right)+(\nabla u, \nabla w)=\left(|u|^{p}, w\right) .
\end{equation}
Letting $w=u$, (\ref{e12}) implies that
\begin{equation}\label{e13}
\frac{p}{p+1} \frac{\mathrm{d}}{\mathrm{d} t}\|u(t)\|_{p+1}^{p+1}+D(u)=0, \quad 0 \leqslant t<\infty .
\end{equation}
From $0<E\left(u_0\right)<d, D\left(u_0\right)>0$ and Theorem \ref{t31}, we have $u(t) \in W_\delta$ for $\delta_1<\delta<\delta_2$ and $0 \leqslant t<\infty$, where $\delta_1<\delta_2$ are the two roots of equation $d(\delta)=E\left(u_0\right)$.
Hence, we obtain $D_{ \delta}(u) \geqslant 0$ for $\delta_1<\delta<\delta_2$ and $D_{\delta_1}(u) \geqslant 0$ for $0 \leqslant t<\infty$.
So, (\ref{e13}) gives
\begin{equation}\label{e13b}
\frac{p}{p+1} \frac{\mathrm{d}}{\mathrm{d} t}\|u\|_{p+1}^{p+1}+\left(1-\delta_1\right)A(u)+D_{\delta_1}(u)=0, \quad 0 \leqslant t<\infty .
\end{equation}
Now (\ref{e13b}) and \eqref{e03} imply that
\begin{equation*}
\frac{p}{p+1} \frac{\mathrm{d}}{\mathrm{d} t}\|u(t)\|_{p+1}^{p+1}+S\left(1-\delta_1\right)\|u\|_{p+1}^2 \leq 0, \quad 0 \leqslant t<\infty,
\end{equation*}
By Lemma \ref{lf44}, we have

$$
\|u\|_{p+1}^{p+1}  \leq
\begin{cases}
\left(B_0^{1-\gamma}-C(1-\gamma) t\right)^{\frac{1}{1-\gamma}}, & 0 \leq t<t^{*}, \\
0, & t \geq t^{*},
\end{cases}
$$
where $\gamma=\frac{2}{p+1}$,  $B(0)= \left\|u(0)\right\|_{p+1}^{p+1}, C=S\left(1-\delta_1\right)\frac{(p+1)}{p}$, $t^{*} = \displaystyle\frac{B_0^{1-\gamma}}{C(1-\gamma)}$.
This completes the proof.
\end{proof}

\begin{lemma} \label{lf44}
Let $C>0$ and $0<\gamma<1$, and let $B_0 = B(0) \geq 0$. Suppose the nonnegative function $B(t)$ satisfies the differential inequality
$$
\frac{\mathrm{d}}{\mathrm{d}t}B(t)+C[B(t)]^{\gamma}\leq 0, \quad 0 \leq t<\infty.
$$
Then
$$
B(t) \leq
\begin{cases}
\left(B_0^{1-\gamma}-C(1-\gamma) t\right)^{\frac{1}{1-\gamma}}, & 0 \leq t<t^{*}, \\
0, & t \geq t^{*},
\end{cases}
$$
where $$t^{*} = \displaystyle\frac{B_0^{1-\gamma}}{C(1-\gamma)}.$$
\end{lemma}

\begin{proof}
First, if $B(0)=0$, then $B(t)\equiv 0$ for all $t\in(0,+\infty)$.

Next, if $B(0)>0$, we assert that there exists $t^{*}>0$ such that $B(t)>0$ for $t\in(0,t^{*})$ and $B(t)=0$ for $t\in[t^{*},\infty)$.
If not, then $B(t)>0$ for all $t>0$. From the inequality we have
$$
\frac{\mathrm{d}B(t)}{[B(t)]^{\gamma}}\leq -C\mathrm{d}t, \quad 0 \leq t<\infty.
$$
Integrating both sides from $0$ to $t$, we obtain
$$
\int_0^t\frac{\mathrm{d}B(s)}{[B(s)]^{\gamma}}\leq \int_0^t -C\,\mathrm{d}s, \quad 0 \leq t<\infty.
$$
This implies
$$
0<B(t)^{1-\gamma} \leq B_0^{1-\gamma}-C(1-\gamma) t.
$$
Define
$$
t^{*}=\frac{B_0^{1-\gamma}}{C(1-\gamma)}.
$$
Therefore
$$
B(t) \leq
\begin{cases}
\left(B_0^{1-\gamma}-C(1-\gamma) t\right)^{\frac{1}{1-\gamma}}, & 0 \leq t<t^{*}, \\
0, & t \geq t^{*}.
\end{cases}
$$
\end{proof}

\section{Critical initial energy $E\left(u_0\right)=d$}

The goal of this section is to prove Theorem \ref{t41}, Theorem \ref{t42} and Theorem \ref{tt43}.
\begin{theorem}\label{t41}
 (Global existence). Assume that $u_0 \in X, E\left(u_0\right)=d$ and $D\left(u_0\right) \geqslant 0$. Then problem \eqref{e01} has a global weak solution $u(t) \in L^{\infty}\left(0, \infty ; X\right)$ and $u(t) \in \bar{W}=W \cup \partial W$ for $0 \leqslant t<\infty$.
\end{theorem}

\begin{lemma}\label{l41}
Assume that $u \in X,\|\nabla u\|_{2}^2 \neq 0$, and $D(u) \geq 0$. Then
  \begin{description}
    \item[(1)]  $\lim _{\mu \rightarrow 0} E(\lambda u)=0, \lim _{\mu \rightarrow+\infty} E(\mu u)=-\infty$;
    \item[(2)]  On the interval $0<\mu<\infty$, there exists a unique $\mu^*=\mu^*(u)$, such that

\begin{equation*}
\left.\frac{\mathrm{d}}{\mathrm{d} \mu} E(\mu u)\right|_{\mu=\mu^*}=0;
\end{equation*}

    \item[(3)] $E(\mu u)$ is increasing on $0 \leqslant \mu \leqslant \mu^*$, decreasing on $\mu^* \leqslant \mu<\infty$ and takes the maximum at $\mu=\mu^*$;
    \item[(4)]  $D(\mu u)>0$ for $0<\mu<\mu^*, D(\mu u)<0$ for $\mu^*<\mu<\infty$, and $D\left(\mu^* u\right)=0$.
\end{description}
\end{lemma}

\begin{proof}
  (1) Firstly, from the definition of $E(u)$, i.e.
$$
E(u)=\frac{1}{2}A(u)-\frac{1}{p+1}B(u)
$$
and we see that
$$
E(\mu u)=\frac{1}{2}A(\mu u)-\frac{1}{p+1}B(\mu u).
$$
Hence, we have
\begin{equation*}
\lim _{\mu \rightarrow 0} E(\mu u)=0 \quad \text { and } \quad \lim _{\mu \rightarrow+\infty} E(\mu u)=-\infty .
\end{equation*}

(2) It is easy to show that
$$
\frac{\mathrm{d}}{\mathrm{d }\mu} E( u)=\mu A(u)-\mu^p B(u),
$$
which leads to the conclusion.

(3) By Lemma \ref{l41} (2), one has
$$
\begin{array}{ll}
\frac{\mathrm{d}}{\mathrm{d} \mu} E(\mu u)>0 & \text { for } 0<\mu<\mu^*, \\
\frac{d}{\mathrm{d} \mu} E(\mu u)<0 & \text { for } \mu^*<\mu<\infty,
\end{array}
$$
which leads to the conclusion.

(4) The conclusion follows from

$$
D(\mu u)=\frac{\mathrm{d}}{\mathrm{d} \mu} E(\mu u)=\mu A(u)-\mu^p B(u).
$$
As desired.
\end{proof}

\begin{proof}[Proof of Theorem \ref{t41}]
 Firstly, $E\left(u_0\right)=d$ implies that $\left\|u_0\right\|_{2} \neq 0$. Choose a sequence $\left\{\mu_m\right\}$ such that $0<\mu_m<1, m=1,2, \ldots$ and $\mu_m \rightarrow 1$ as $m \rightarrow \infty$. Let $u_{0 m}=\mu_m u_0$. We consider the following initial problem
\begin{equation}\label{e14}
 \begin{cases}
& \dfrac{\partial}{\partial t} u^{p}=\Delta u+u^{p} \quad \text { in } \Omega \times(0, T), \\
&\quad\ \ \  u(x,0)=u_{0m}, \quad \text { in }  \Omega,\\
&\quad\ \ \  u=0, \quad \text { on } \partial \Omega \times(0, T),
\end{cases}
\end{equation}
From $D\left(u_0\right) \geqslant 0$ and Lemma \ref{l41}, we have $\mu^*=\mu^*\left(u_0\right) \geqslant 1$.
Thus, we get $D\left(u_{0 m}\right)= D\left(\mu_m u_0\right)>0$ and $E\left(u_{0 m}\right)=E\left(\mu_m u_0\right)<E\left(u_0\right)=d$. From Theorem \ref{t32},
it follows that for each $m$ problem (\ref{e14}) admits a global weak solution
$u_m(t) \in L^{\infty}\left(0, \infty ;  X\right)$ with $u_{m }(t) \in L^2\left(0, \infty ;  X\right)$ and $u_m(t) \in W$ for $0 \leqslant t<\infty$ satisfying

\begin{gather*}
\left(u_{m,t}, w\right)+\left(\nabla u_{m}, \nabla w\right)=\left(|u_m|^{p}, w\right), \text { for all } w \in  X, t>0,  \\
\int_{0}^{t}\int_{\Omega} pu_{m}^{p-1}u^2_{m,\tau}\mathrm{d}x\mathrm{d}\tau +E\left(u_m(t)\right)=E\left(u_{0 m}\right)<d, \quad 0 \leqslant t<\infty,
\end{gather*}
which implies that
$$
\begin{aligned}
E\left(u_m\right) & =\frac{1}{2}A(u_m)-\frac{1}{p+1}B(u_m) \\
& =\frac{p-1}{2(p+1)}A(u_m)+\frac{1}{p+1} D\left(u_m\right).
\end{aligned}
$$
Therefore, one has
\begin{equation*}
 \int_{0}^{t}\int_{\Omega} pu_{m}^{p-1}u^2_{m,\tau}\mathrm{d}x\mathrm{d}\tau+\frac{p-1}{2(p+1)}A(u_m)<d, \quad 0 \leqslant t<\infty .
\end{equation*}
The remainder of the proof is similar to the proof of Theorem \ref{t32}.
\end{proof}
\begin{theorem}\label{t42}
 (Blow-up). Assume that $u_0 \in X, E\left(u_0\right)=d$ and $D\left(u_0\right)<0$.
 Then the weak solution $u(t)$  for problem (\ref{e01})  exists globally for all $t>0$ and undergoes infinite-time blow-up, i.e.,
$$
\|u(t)\|_{p+1}^{p+1} \rightarrow+\infty, \quad t \rightarrow+\infty.
$$
\end{theorem}
\begin{proof}[Proof of Theorem \ref{t42}]
 Let $u(t)$ be any weak solution of problem (\ref{e01}) with $E\left(u_0\right)=d$ and $D\left(u_0\right)<0, S$ be the existence time of $u(t)$.
 We next prove $T<\infty$. We argue by contradiction. Suppose that there would exist a global weak solution $u(t)$. Set
\begin{equation*}
f(t)=\int_0^t\|u\|_{p+1}^{p+1} \mathrm{~d} \tau, t>0.
\end{equation*}
Multiplying (\ref{e01}) by $u$ and integrating over $\Omega \times(0, t)$, we get

$$\begin{aligned}
\int_0^t \int_{\Omega} p u^p u_\tau \mathrm{d} x \mathrm{d }\tau & =\frac{p}{p+1}\left(\|u(t)\|_{p+1}^{p+1}-\|u(0)\|_{p+1}^{p+1}\right) \\
& =\int_0^t \int_{\Omega}\left(u^{p+1}-|\nabla u|^2\right) \mathrm{d} x \mathrm{d} \tau \\
& =\int_0^t(B(u)-A(u)) d \tau
\end{aligned}$$
According to the definition of $f(t)$, we have $f^{\prime}(t)=\|u\|_{p}^p$ and hence
$$
\begin{aligned}
f^{\prime}(t) & =\|u\|_{p+1}^{p+1}=\|u(0)\|_{p+1}^{p+1}+\frac{p+1}{p} \int_0^t \int_{\Omega}\left(u^{p+1}-|\nabla u|^2\right) \mathrm{d} x \mathrm{d} \tau \\
& =(p+1) \int_0^t \int_{\Omega} u^p u_\tau \mathrm{d} x \mathrm{d} \tau+\|u(0)\|_{p+1}^{p+1}
\end{aligned}
$$
and
\begin{equation}\label{e15}
\frac{p+1}{p}(B(u)-A(u))=-\frac{p+1}{p} D(u).
\end{equation}
By an argument similar to the proof of Theorem \ref{t33}, we also obtain
\begin{equation}\label{e17}
\begin{gathered}
f(t) f^{\prime \prime}(t)-\left(f^{\prime}(t)\right)^2 \geq \frac{p^2-1}{2 p} S\left(f^{\prime}(t)\right)^{\frac{2}{p+1}} f(t)-\frac{(p+1)^2}{p} \int_0^t\|u\|_{p+1}^{p+1} \mathrm{~d} \tau E(u(0)) \\
-2\|u(0)\|_{p+1}^{p+1} f^{\prime}(t)
\end{gathered}
\end{equation}

On the other hand, from $E\left(u_0\right)=d>0, D\left(u_0\right)<0$ and the continuity of $E(u)$ and $D(u)$ with respect to $t$,
it follows that there exists a sufficiently small $t_1>0$ such that $E\left(u\left(t_1\right)\right)>0$ and $D(u(t))<0$ for $0 \leqslant t \leqslant t_1$.
 Hence $\left(u_t, pu^p\right)=-D(u)>0$,$\int_0^tp u^{p-1} u_\tau^2 \mathrm{d} \tau>0$ for $0 \leqslant t \leqslant t_1$.
 So, using the continuity of $\int_0^tp u^{p-1} u_\tau^2 \mathrm{d} \tau$, we can choose a $t_1$ such that
\begin{equation*}
0<d_1=d-\int_0^{t_1}\int_{\Omega}p u^{p-1} u_\tau^2\mathrm{d}x\mathrm{d} \tau<d.
\end{equation*}
And by (\ref{e08}), we get
\begin{equation*}
0<E\left(u\left(t_1\right)\right)=d-\int_0^{t_1}\int_{\Omega}p u^{p-1} u_\tau^2 \mathrm{d} \tau=d_1<d.
\end{equation*}
So we can choose $t=t_1$ as the initial time, then we obtain $u(t) \in V_\delta$
for $\delta \in\left(\delta_1, \delta_2\right), t_1 \leqslant t<\infty$,
 where $\left(\delta_1, \delta_2\right)$ is the maximal interval including
  $\delta=1$ such that $d(\delta)>d_1$ for $\delta \in\left(\delta_1, \delta_2\right)$.
  Thus we get $D_{ \delta}(u)<0$ and $\sqrt{A(u)}>r(\delta)$ for $\delta \in\left(1, \delta_2\right), t_1 \leqslant t<\infty$,
  and $D_{ \delta_2}(u) \leqslant 0, \sqrt{A(u)} \geqslant r\left(\delta_2\right)$
  for $t_1 \leqslant t<\infty$.  We then proceed along lines analogous to the proof of Theorem \ref{t33}, which completes the proof of the theorem.

\end{proof}

\begin{theorem}\label{tt43}
 Assume that $u_0 \in X, E\left(u_0\right)=d$ and $D\left(u_0\right)>0, \delta_1<\delta_2$ are the two roots of equation
 $d(\delta)=E\left(u_0\right)$. Then, for the global weak solution $u$ of problem (\ref{e01})
undergoes finite-time extinction, that is,
\begin{equation}\label{e0538}
 \|u\|_{p+1}^{p+1}  \leq
\begin{cases}
\left(B_0^{1-\gamma}-C(1-\gamma) t\right)^{\frac{1}{1-\gamma}}, & 0 \leq t<t^{*}, \\
0, & t \geq t^{*},
\end{cases}
\end{equation}
where $\gamma=\frac{2}{p+1}$,  $B(0)= \left\|u(0)\right\|_{p+1}^{p+1},
C=S\left(1-\delta_1\right)\frac{(p+1)}{p}$,
$t^{*} = \displaystyle\frac{B_0^{1-\gamma}}{C(1-\gamma)}$.
\end{theorem}
\begin{proof}
 We first know that problem (\ref{e01}) has a local  weak solution from Theorem \ref{t21}. Furthermore,
 Using Theorem \ref{t42} and (\ref{e08}), if $u(t)$ is a local weak solution of problem (\ref{e01})
 with $E\left(u_0\right)=d, D\left(u_0\right)>0$, then must have $D(u) \geq 0$ for $0 \leq t<+\infty$. Next, we distinguish two cases:

(1) Suppose that $D(u)>0$ for $0 \leqslant t<\infty$. Multiplying (\ref{e01}) by $w, w \in L^{\infty}\left(0, \infty ; X\right)$, we have

\begin{equation}\label{e18}
\left(pu^{p-1}u_t, w\right)+\left(\nabla u, \nabla w\right)=\left(|v|^{p}, w\right), \text { for all } w \in X, t>0 .
\end{equation}
Letting $w=u$, (\ref{e18}) implies that
\begin{equation}\label{e19}
\frac{p}{p+1} \frac{\mathrm{d}}{\mathrm{d} t}\|u\|_{p+1}^{p+1}=-D(u)<0, \quad 0 \leqslant t<\infty.
\end{equation}
Since $\left\|u_t\right\|_{p+1}^{p+1}>0$, we have that $\int_0^t\left\|u_\tau\right\|_{p+1}^{p+1} \mathrm{d} \tau$
 is increasing for $0 \leqslant t<\infty$. By choosing any $t_1>0$ such that
\begin{equation}\label{e20}
0<d_1=d-\int_0^{t_1}\int_{\Omega}pu^{p-1}u^2_{\tau} \mathrm{d}x \mathrm{d} \tau<d.
\end{equation}
From (\ref{e08}), if follows that $0<E(u) \leq d_1<d$, and $u(t) \in W_\delta$
for $\delta_1<\delta<\delta_2$ and $0 \leqslant t<\infty$, where $\delta_1<\delta_2$ are the two roots of equation $d(\delta)=E\left(u_0\right)$.
Hence, we obtain $D_{\delta_1}(u) \geqslant 0$ for $\delta_1<\delta<\delta_2$ and $D_{\delta_1}(u) \geqslant 0$ for $t_1 \leqslant t<\infty$. So, (\ref{e19}) gives
$$
\frac{p}{p+1} \frac{\mathrm{d}}{\mathrm{d} t}\|u\|_{p+1}^{p+1}+\left(1-\delta_1\right)\left(A(u)\right)+D_{\delta_1}(u)=0, \quad 0 \leqslant t<\infty.
$$
Repeating the argument from Theorem \ref{t34}, we get

\begin{equation}\label{e0538}
 \|u\|_{p+1}^{p+1}  \leq
\begin{cases}
\left(B_0^{1-\gamma}-C(1-\gamma) t\right)^{\frac{1}{1-\gamma}}, & 0 \leq t<t^{*}, \\
0, & t \geq t^{*},
\end{cases}
\end{equation}
where $\gamma=\frac{2}{p+1}$,  $B(0)= \left\|u(0)\right\|_{p+1}^{p+1},
C=S\left(1-\delta_1\right)\frac{(p+1)}{p}$,
$t^{*} = \displaystyle\frac{B_0^{1-\gamma}}{C(1-\gamma)}$.

(2) Suppose that there exists a $t_1>0$ such that $D\left(u\left(t_1\right)\right)=0$ and $D(u)>0$ for $0 \leqslant t<t_1$.
 Then, $ \int_{\Omega}pu^{p-1}u^2_{\tau} \mathrm{d}x>0$ and
  $\int_0^t \int_{\Omega}pu^{p-1}u^2_{\tau} \mathrm{d}x \mathrm{d} \tau$
  is increasing for $0 \leqslant t<t_1$. By (\ref{e20}) we have
\begin{equation*}
E\left(u\left(t_1\right)\right)=d-\int_0^{t_1} \int_{\Omega}pu^{p-1}u^2_{\tau} \mathrm{d}x\mathrm{d} \tau<d
\end{equation*}
and $\left\|u\left(t_1\right)\right\|_{p+1}=0$. Then, we have that $u(t) \equiv 0$ for $t_1 \leqslant t<\infty$.
Hence, the proof is complete.

\end{proof}

\section{High initial energy $E\left(u_0\right)>d$}

In this section, we investigate the conditions to ensure the existence of global solutions or
blow-up solutions to problem \eqref{e01} with $E\left(u_0\right)>d$.
\begin{lemma}\label{l51}
For any $\alpha>d, \lambda_\alpha$ and $\Lambda_\alpha$ defined in section 1  satisfy
\begin{equation*}
0<\lambda_\alpha \leq \Lambda_\alpha<+\infty.
\end{equation*}
\end{lemma}

\begin{proof}
  (1)Since $u \in \mathcal{N}$, we have $A(u)=B(u).$ By \eqref{e03}, we see that
\begin{equation*}
A(u)=B(u) \leqslant\left(\frac{1}{S}\left(A(u)\right)\right)^{\frac{p+1}{2}} .
\end{equation*}
So we have $A(u) \geq S^{\frac{N}{2}}$. Then from Lemma \ref{l07} (1), we have $\lambda_\alpha>0$. From the definition of $\mathcal{N}^\alpha$ and $\Lambda$ and
 $u \in \mathcal{N}$, we have $\Lambda_{\alpha}<+\infty$.

\end{proof}

\begin{theorem}\label{t51}
   Suppose that $E\left(u_0\right)>d$, then we have
\begin{description}
  \item[(1)] If $ u_0 \in \mathcal{N}_{+}$ and $\left\| u_0\right\|_{p+1} \leq \lambda_{E\left(u_0\right)}$, then $u_0 \in \mathcal{G}_0$;
  \item[(2)] If $ u_0 \in \mathcal{N}_{-}$ and $\left\| u_0\right\|_{p+1} \geq \Lambda_{E\left( u_0\right)}$, then $u_0 \in \mathcal{B}$.
\end{description}
\end{theorem}

\begin{proof}
The maximal existence time of the solutions to problem (\ref{e01}) with initial value $u_0$ is denoted by $T_0$.
If the solution is global, i.e. $T\left(u_0\right)=+\infty$, the limit set of $u_0$ is denoted by $\omega_0$.

(1) Suppose that $ u_0 \in \mathcal{N}_{+}$ with $\left\| u_0\right\|_{p+1} \leq \lambda_{E\left(u_0\right)}$.
We firstly prove that $u(t) \in \mathcal{N}_{+}$ for all $t \in \left[0, T\left(u_0\right)\right)$.
 Assume, on the contrary, that there exists a $t_0 \in\left(0, T\left(u_0\right)\right)$
 such that $u(t) \in \mathcal{N}_{+}$ for $0 \leq t<t_0$ and $u\left(t_0\right) \in \mathcal{N}$.
 It follows from $D(u(t))=-\int_{\Omega} pu_t(x, t) u^{p}(x, t) \mathrm{d} x$ that $u_t(x, t) \neq 0$
 for $(x, t) \in \Omega \times\left(0, t_0\right)$.
 According to (\ref{e08}) we then have $E\left(u\left(t_0\right)\right)<E\left(u_0\right)$,
 which implies that $u\left(t_0\right) \in E^{E\left(u_0\right)}$. Therefore, $u\left(t_0\right) \in \mathcal{N}^{E\left(u_0\right)}$.
 Recalling the definition of $\lambda_{E\left(u_0\right)}$, we get

\begin{equation}\label{e21}
\left\| u\left(t_0\right)\right\|_{p+1} \geq \lambda_{E}\left(u_0\right).
\end{equation}
Since $D(u(t))>0$ for $t \in\left[0, t_0\right)$, we obtain from (\ref{e13}) that
\begin{equation*}
\left\|u\left(t_0\right)\right\|_{p+1}<\left\| u_0\right\|_{p+1} \leq \lambda_{E\left(u_0\right)},
\end{equation*}
which contradicts (\ref{e21}). Hence, $u(t) \in \mathcal{N}_{+}$ which shows that $u(t) \in E^{E\left(u_0\right)}$ for all $t \in \left[0, T\left(u_0\right)\right)$.
Now Lemma \ref{l07} (2) implies that the orbit $\{u(t)\}$ remains bounded in $X$
for $t \in\left[0, T\left(u_0\right)\right)$ so that $T\left(u_0\right)=\infty$. Assume that $\omega$ is an arbitrary element in $\omega\left(u_0\right)$.
Then by (\ref{e08}) and (\ref{e13}) we obtain
\begin{equation*}
\| \omega\|_{p+1}<\lambda_{E\left(u_0\right)}, \quad E(\omega)<E\left(u_0\right),
\end{equation*}
which, according to the definition of $\lambda_{E\left(u_0\right)}$ again, implies that $\omega\left(u_0\right) \cap \mathcal{N}=\varnothing$. So, $\omega\left(u_0\right)= \{0\}$, i.e. $u_0 \in \mathcal{G}_0$.

(2) Suppose that $u_0 \in \mathcal{N}_{-}$with $\left\| u_0\right\|_{p+1} \geq \Lambda_{E\left(u_0\right)}$.
 We now prove that $u(t) \in \mathcal{N}_{-}$ for all $t \in \left[0, T\left(u_0\right)\right)$.
 Assume, on the contrary, that there exists a $t^0 \in\left(0, T\left(u_0\right)\right)$ such that
 $u(t) \in \mathcal{N}_{-}$ for $0 \leq t<t^0$ and $u\left(t^0\right) \in \mathcal{N}$.
 Similarly to case (1), one has $E\left(u\left(t^0\right)\right)<E\left(u_0\right)$,
  which implies that $u\left(t^0\right) \in E^{E\left(u_0\right)}$. Therefore, $u\left(t^0\right) \in \mathcal{N}^{E\left(u_0\right)}$.
 Recalling the definition of $\Lambda_{E\left(u_0\right)}$, we infer

\begin{equation}\label{e22}
\left\|u\left(t^0\right)\right\|_{p+1} \leq \Lambda_{E\left(u_0\right)} .
\end{equation}
On the other hand, from (\ref{e13}) and the fact that $D(u(t))<0$ for $t \in\left[0, t^0\right)$, we obtain
\begin{equation*}
\left\| u\left(t^0\right)\right\|_{p+1}>\left\| u_0\right\|_{p+1} \geq \Lambda_{E\left(u_0\right)},
\end{equation*}
which contradicts (\ref{e22}).

Assume that $T\left(u_0\right)=\infty$. Then for each $\omega \in \omega\left(u_0\right)$, it follows from by (\ref{e08}) and (\ref{e13}) that

\begin{equation*}
\| \omega\|_{p+1}>\Lambda_{E\left(u_0\right)}, \quad E(\omega)<E\left(u_0\right).
\end{equation*}
Noting the definition of $\Lambda_{E\left(u_0\right)}$ again, we have $\omega\left(u_0\right) \cap \mathcal{N}=\varnothing$. Hence,
it  holds that $\omega\left(u_0\right)=\{0\}$, which contradicts Lemma \eqref{l07} (1). Therefore, $T\left(u_0\right)<\infty$.
This ends the proof.

\end{proof}

\begin{theorem}\label{t52}
Assume that $u_0 \in X$ satisfies
\begin{equation}\label{e23}
  \frac{2(p+1)}{p-1}E\left(u_0\right) \leq\frac{p-1}{2(p+1)}\left\|u_0\right\|_{p+1}^{p+1}.
\end{equation}
Then, $u_0 \in \mathcal{N}_{-} \cap \mathcal{B}$.
\end{theorem}
\begin{proof}
 Firstly, we observe
$$
\begin{aligned}
E\left(u_0\right) & =\frac{1}{2}A(u_0)-\frac{1}{p+1}B(u_0) \\
& =\frac{1}{2} D\left(u_0\right)+\frac{p-1}{2p+2}B(u_0).
\end{aligned}
$$
Thus, \eqref{e23} implies that
$$
E\left(u_0\right)-\frac{p-1}{2p+2}B(u_0)=\frac{1}{2} D\left(u_0\right)<0,
$$
which shows that $u_0 \in \mathcal{N}_{-}$. Then for any $u \in \mathcal{N}^{E\left(u_0\right)}$, by the definition of $\mathcal{N}^{E\left(u_0\right)}$
  and  (\ref{e23}) one has
\begin{align*}
 \|u\|_{p+1}^{p+1}&= B(u)=A(u)  \leq \frac{2 E\left(u_0\right) (p+1)}{p-1}\\
&\leq \frac{p-1}{2(p+1)}\left\|u_0\right\|_{p+1}^{p+1}\leq \|u_0\|_{p+1}^{p+1}.
\end{align*}
Taking supremum over $\mathcal{N}^{E\left(u_0\right)}$  and by Theorem \ref{t51} we can deduce
$$
\left\|u_0\right\|_{p+1} \geq \Lambda_{E\left(u_0\right)} .
$$
Thus, $u_0\in \mathcal{N}_{-} \cap \mathcal{B}$. This finishes the proof.
\end{proof}

\section{Asymptotic behavior and  stationary solutions}

In this section,  we consider the asymptotic behavior of the global solutions, which is similar to
the Palais-Smale  sequence of stationary equation
\begin{equation}\label{state}
 \begin{cases}
& -\Delta u=u^{p}  \text { in } \Omega, \\
&  u=0, \quad \text { on } \partial \Omega.
\end{cases}
\end{equation}

\begin{theorem}\label{t61}
Let $u\left(x, t ; u_0\right)$ be a global solution of the problem (\ref{e01}) and uniformly bounded in $X$ with respect to $t$.
Then, for any subsequence $t_n \rightarrow \infty$,
there exists a stationary solution $w$ such that $u\left(x, t_n ; u_0\right) \rightharpoonup w$ in $X$.
\end{theorem}

\begin{theorem}\label{t62}
 Let $u\left(x, t ; u_0\right)$ be a global solution of the problem (\ref{e01}). Then, its $\omega$-limit contains a stationary solution $w$.
 The $\omega$-limit set is defined as
$$
\omega\left(u_0\right)=\left\{w \in X \mid \exists t_n \rightarrow+\infty, u\left(x, t_n ; u_0\right) \rightharpoonup w \text { in } X\right\}.
$$
\end{theorem}

\begin{theorem}\label{t63}
Let $u(t)$ be a positive solution of \eqref{e01}.
Assume that no subsequence of $\{u(t)\}$ converges strongly in $H_0^1(\Omega)$ as $t\to\infty$.
Then there exist a sequence $\{t_m\}$ with $t_m\to\infty$ and a steady state $u_\infty\in H_0^1(\Omega)$ of \eqref{state} such that
\[
u(t_m) \rightharpoonup u_\infty \quad \text{weakly in } H^1(\Omega).
\]
Moreover, there exists a nonnegative integer $k$, points $x_m^j\in\Omega$ and positive real numbers $R_m^j\in\mathbb{R}^+$ for $j=1,2,\dots,k$, such that the following hold:
\[
u(t_m)-u_\infty-\sum_{j=1}^k \big(R_m^j\big)^{-\frac{2}{p-1}} U_j\left(\frac{x-x_m^j}{R_m^j}\right)
\to 0 \quad \text{in } H^1(\mathbb{R}^N) \text{ as } m\to\infty,
\]
\[
\lim_{m\to\infty} E\big(u(t_m)\big) = E(u_\infty) + \sum_{j=1}^k E_{\mathbb{R}^N}(U_j),
\]
and
\[
R_m^j \operatorname{dist}\big(x_m^j, \partial\Omega\big) \to +\infty.
\]
Here each $U_j$ is a standard bubble for problem \eqref{state}, and the energy functional $E_{\mathbb{R}^N}$ is defined by
\[
E_{\mathbb{R}^N}(u)=\frac{1}{2} \int_{\mathbb{R}^N} |\nabla u|^2 \mathrm{d}x - \frac{1}{p+1} \int_{\mathbb{R}^N} u^{p+1} \mathrm{d}x.
\]
\end{theorem}

\noindent\textbf{Remark.} Theorem \ref{t63} can be directly deduced from the work of M. Struwe  \cite{Struwe}.

\begin{proof}[Proof of Theorem \ref{t61}]
 Let us denote $u_n:=u\left(x, t_n\right)$. Since $\left\{u_n\right\}$ is uniformly bounded in $X$,
 then   there exists a subsequence (here we still denote
 by $\left\{u_n\right\}$ ) and a function $w\in  X$ such that

$$
\begin{aligned}
u_n \rightharpoonup w & \text { in } X, \\
u_n \rightarrow w & \text { in } L^2(\Omega), \\
u_n \rightharpoonup w & \text { in } L^{p+1}\left(\Omega\right), \\
u_n \rightarrow w & \text { a.e. in } \mathbb{R}^N.
\end{aligned}
$$
Let $U_n:=u\left(t_n+t\right)$ for $t \in(0,1)$. Clearly, $U_n$ is uniformly bounded in $X$, we show

$$
U_n \rightharpoonup w \text { in } L^{p+1}(\Omega).
$$
Indeed, for $t \in(0,1)$, by (\ref{e10}), we have
\begin{equation}
\int_{0}^{+\infty}\int_{\Omega} pu^{p-1}u^2_{\tau}\mathrm{d}x\mathrm{d}\tau +E(w))=E\left(u_0\right)<\infty, \quad 0 \leqslant t<T,
\end{equation}
which means $u^{\frac{p-1}{2}}u_t\in L^2(\Omega)$. So
$$
\int_{\Omega}\left|U^\frac{p+1}{2}_n-u^\frac{p+1}{2}_n\right|^2 \mathrm{d} x
\leq \frac{(p+1)^2}{4}t \int_{t_n}^{t+t_n} \int_{\Omega} u^{p-1}u_\tau^2\mathrm{d} x \mathrm{d} \tau \rightarrow 0
$$
for $0 \leq t \leq 1$ as $t_n \rightarrow \infty$, which implies that
 $$\left\|u^\frac{p+1}{2}\left(t+t_n\right)-u^\frac{p+1}{2}\left(t_n\right)\right\|_{2} \rightarrow 0\ \ \text{as}\ t_n \rightarrow \infty,$$
 for $0 \leq t \leq 1$. Therefore, we have
\begin{equation}\label{sh1}
 U^{\frac{p+1}{2}}_n \rightharpoonup w^{\frac{p+1}{2}} \ \text { in }  L^{2}(\Omega),
\end{equation}
\begin{equation}\label{sh1}
 U_n \rightharpoonup w \ \text { in }  L^{p+1}(\Omega),
\end{equation}
\begin{equation}\label{sh4}
 U_n^{p} \rightharpoonup w^{p}\  \text { in } \  L^\frac{p+1}{p}\left(\Omega\right),
\end{equation}
\begin{equation}\label{sh2}
  U_n \rightarrow w \ \text { a.e. in } \Omega.
\end{equation}

In order to show that $w$ is a stationary solution, we pass to the limit
(as $t_n \rightarrow \infty$ ) in  Definition \ref{wd} with a suitably chosen test function. Let

$$
\phi(x, t)=
\begin{cases}
\rho\left(t-t_n\right) \Psi(x), & t>t_n, y \in \Omega, \\
0, & 0 \leq t \leq t_n, x \in  \Omega,
\end{cases}
$$
where
$$
\Psi \in X, \rho \in C_0^2(0,1), \rho \geq 0, \int_0^1 \rho(t) d t=1.
$$
Take $\phi$ as test function in Definition \ref{wd}, we have
$$\begin{aligned}
&\int_{t_n}^{t_n+1} \int_{\Omega}\left[u^p \rho^{\prime}\left(t-t_n\right)
\Psi-\rho\left(t-t_n\right) \nabla u \nabla \Psi+u^{p} \rho\left(t-t_n\right) \Psi\right] \mathrm{d} x \mathrm{~d} \tau=0.
\end{aligned}$$
The transformation $\tau=t-t_n$ leads to
\begin{equation}\label{e61}
 \begin{aligned}
&\int_{0}^{1} \int_{\Omega}\left[u^p(\tau+t_n) \rho^{\prime}\left(\tau\right)
\Psi-\rho\left(\tau\right) \nabla u(\tau+t_n) \nabla \Psi+(u(\tau+t_n))^{p}\rho\left(\tau\right) \Psi\right] \mathrm{d} x \mathrm{~d}\tau=0.
\end{aligned}
\end{equation}
Now we rewrite Equation (\ref{e61}) as follows:

\begin{align*}
 &\int_0^1 \int_{\Omega}\left[u\left(t_n\right) \rho^{\prime}(\tau) \Psi-\rho(\tau)
 \nabla u\left(t_n\right) \nabla \Psi+(u\left(t_n\right))^{p} \rho(\tau) \Psi\right]\mathrm{d}x\mathrm{d}\tau\\
& \quad+\int_0^1 \int_{\Omega}\left[u\left(t_n+\tau\right)-u\left(t_n\right)\right] \rho^{\prime}(\tau) \Psi \mathrm{d} y \mathrm{~d} \tau \\
& \quad-\int_0^1 \int_{\Omega}\left[\nabla u\left(t_n+\tau\right)-\nabla u\left(t_n\right)\right] \rho(\tau) \nabla \Psi  \mathrm{d} y \mathrm{~d} \tau \\
& \quad+\int_0^1 \int_{\Omega}\left[(u\left(t_n+\tau\right))^{p}-(u\left(t_n\right))^{p}\right] \rho(\tau) \Psi  \mathrm{d} x\mathrm{~d} \tau=0.
\end{align*}
By the dominated convergence theorem  and  \eqref{sh1}-\eqref{sh4}
we have
$$
\begin{aligned}
 \int_0^1 \int_{\Omega}\left[\rho(\tau) \nabla u\left(t_n\right) \nabla \Psi-\left(u\left(t_n\right)\right)^{p}\rho(\tau) \Psi\right]\mathrm{d}x\mathrm{d}\tau=
o(1), \ \text { as } n  \rightarrow \infty .
\end{aligned}
$$
From the choice of $\rho$, we obtain
$$
\begin{aligned}
\int_0^1 \int_{\Omega}\left[ \nabla u\left(t_n\right) \nabla \Psi-\left(u\left(t_n\right)\right)^{p}  \Psi\right]\mathrm{d}x\mathrm{d}\tau
=o(1), \ \text { as } n  \rightarrow \infty,
\end{aligned}
$$
which completes the proof of Theorem \ref{t61}.
\end{proof}

\begin{proof}[Proof of Theorem \ref{t62}]
 Let $u=u(t, x)$ be a global solution of problem (\ref{e01}). Then, we have

\begin{equation}\label{e63}
\int_0^{\infty} \int_{\Omega}  pu^{p-1} u_\tau^2 \mathrm{d} x \mathrm{d} \tau \leq C<\infty.
\end{equation}
And hence, there exists a sequence $\left\{t_n\right\}$ satisfying $t_n \rightarrow \infty$ as $n \rightarrow \infty$ such that
\begin{equation}\label{e64}
\int_{\Omega} pu^{p-1}\left|u_\tau\left(t_n, x\right)\right|^2 \mathrm{d} x \rightarrow 0, \text { as } n \rightarrow \infty.
\end{equation}
Indeed, on the contrary, if there exists $c>0$ such that $\int_{\Omega} pu^{p-1} \left|u_\tau\left(t_n, x\right)\right|^2 d x>c$ as $n \rightarrow \infty$,
then, we can derive a contradiction with (\ref{e63}).

Next, letting $u_n:=u\left(t_n, x\right)$,  we have $E(u(t_n))>0$ for $t_n>0$.
Indeed, on the contrary,  if $E(u(t_n))\leq 0$, from \eqref{jiu} and  the proof of \eqref{sue},  we obtain $D(u(t_n))< 0$.
Hence, Theorem \ref{t33} implies that $u(t_n)$  blows up  in infinite time. This is a contradiction. So, $E(u(t_n))>0$.

Then, by (\ref{e10}), we easily know that

\begin{equation}\label{e65}
0<E\left(u\left(t_n\right)\right) \leq E\left(u_0\right).
\end{equation}
Then, (\ref{e64}) and (\ref{e65}) imply that $u_n:=u\left(t_n, x\right)$ is a Palais-Smale
sequence  related to the stationary equation of problem (\ref{e01}). Similar to the standard proof
of boundedness for  Palais-Smale sequence of elliptic equation
(see \cite{hr13}),
it is easy to prove that there exists a constant $C$ such that

$$
\int_{\Omega}\left|\nabla u_n\right|^2 \mathrm{d} x \leq C,
$$
and then there exists a subsequence (denote still by $\left\{v_n\right\}$ ) and a function $w$ such that
\begin{equation}\label{fin}
 \begin{aligned}
&u_n \rightharpoonup w,  \text { in } X, \\
& u_n \rightarrow w,  \text { in } L^q(\Omega)\left(2 \leq q<2^*\right),\\
&u_n \rightharpoonup w,  \text { in } \  L^{p+1}\left(\Omega\right),\\
& |u_n|^{2^{\ast}-1} \rightharpoonup |w|^{2^{\ast}-1},  \text { in } \  L^\frac{2 N}{N+2}\left(\Omega\right),
\end{aligned}
\end{equation}
Furthermore, using \eqref{e65} and \eqref{fin}, one has $u_n \rightharpoonup w$ in $X$,  which means that $w$ is a  stationary solution.
\end{proof}

\section{Pohozaev identity and  nonexistence theorem}
In this section, we derive the Pohozaev identity and establish nonexistence result for  equation \eqref{e01}.
In \cite{r4}, Pohozaev identity is established for semilinear parabolic equation.
\begin{theorem}\label{t81}
Let $u (x, t)$ be a solution to (\ref{e01}). Then it satisfies
\begin{align}\label{po10}
& \int_{\Omega} \frac{|x|^2}{2}pu^{p-1}\left|u_t\right|^2 \mathrm{~d} x+\int_{\Omega} \frac{|x|^2}{2}
 \frac{\mathrm{~d}}{\mathrm{~d} t}\left(\frac{|\nabla u|^2}{2}-\frac{1}{p+1}u^{p+1}\right) \mathrm{d} x \nonumber\\
 & =\int_{\partial \Omega} \frac{|\nabla u|^2}{2}(x \cdot \nu) \mathrm{d} S+\frac{N-2}{2} \int_{\Omega}|\nabla u|^2 \mathrm{~d} x
 -N \int_{\Omega} \frac{1}{p+1}u^{p+1}\mathrm{d} x.
\end{align}
\end{theorem}
\begin{theorem}\label{t82}
 Assume that $\Omega$ is a strictly star-shaped bounded smooth domain with respect to the origin.
  Let $u (x, t)$ be a solution to problem (\ref{e01}) and  satisfy
\begin{align}\label{po11}
& \int_{\Omega} \frac{|x|^2}{2}pu^{p-1}\left|u_t\right|^2 \mathrm{~d} x+\int_{\Omega} \frac{|x|^2}{2}
 \frac{\mathrm{~d}}{\mathrm{~d} t}\left(\frac{|\nabla u|^2}{2}-\frac{1}{p+1}u^{p+1}\right) \mathrm{d} x \nonumber\\
 &\quad -\frac{N-2}{2} \int_{\Omega}|\nabla u|^2 \mathrm{~d} x
 +N \int_{\Omega} \frac{1}{p+1}u^{p+1}\mathrm{d} x\leq 0,
\end{align}
then $u\equiv 0$.
\end{theorem}

\begin{proof}[Proof of Theorem \ref{t81}]
First, multiplying  both sides of the problem \eqref{e01} by $\frac{|x|^2}{2} u_t$ and integrating  over $\Omega$, we get

\begin{equation}\label{po1}
\int_{\Omega} \frac{|x|^2}{2}pu^{p-1}\left|u_t\right|^2 \mathrm{~d} x-\int_{\Omega} \frac{|x|^2}{2} u_t \Delta u \mathrm{~d} x
=\int_{\Omega} \frac{|x|^2}{2} u_t u^p \mathrm{d} x,
\end{equation}
where
\begin{equation}\label{po2}
\int_{\Omega} \frac{|x|^2}{2} u_t u^p \mathrm{d} x
=\int_{\Omega} \frac{|x|^2}{2} \frac{1}{p+1}\frac{\mathrm{~d}}{\mathrm{~d} t} u^{p+1} \mathrm{d} x,
\end{equation}
and by Green formula we obtain
\begin{align}\label{po3}
\int_{\Omega} \frac{|x|^2}{2} u_t \Delta u \mathrm{~d} x & =\int_{\partial \Omega} \frac{|x|^2}{2} u_t \frac{\partial u}{\partial \nu} \mathrm{~d} S-\int_{\Omega} \nabla u \nabla\left(\frac{|x|^2}{2} u_t\right) \mathrm{d} x \nonumber\\
& =0-\int_{\Omega} \nabla u \nabla\left(\frac{|x|^2}{2} u_t\right) \mathrm{d} x  \nonumber\\
& =\int_{\Omega} u_t(x \cdot \nabla u) \mathrm{d} x-\int_{\Omega} \frac{|x|^2}{4} \frac{\mathrm{~d}}{\mathrm{~d} t}|\nabla u|^2 \mathrm{~d} x .
\end{align}
Since $u_t=0$ on $\partial \Omega$, the boundary integral in (\ref{po3}) vanishes.
 Now, from (\ref{po1}), (\ref{po2})  and (\ref{po3}), we have
\begin{align}\label{po4}
\int_{\Omega} u_t(x \cdot \nabla u) \mathrm{d} x&=\int_{\Omega}
\frac{|x|^2}{2}pu^{p-1}\left|u_t\right|^2 \mathrm{~d} x+\int_{\Omega} \frac{|x|^2}{4} \frac{\mathrm{~d}}{\mathrm{~d} t}|\nabla u|^2 \mathrm{~d} x\nonumber\\
&\quad -\int_{\Omega} \frac{|x|^2}{2}\frac{1}{p+1} \frac{\mathrm{~d}}{\mathrm{~d} t} u^{p+1} \mathrm{d} x.
\end{align}
Next, we multiply equation (\ref{e01}) by $x \cdot \nabla u=\sum_{j=1}^N x_j u_{x_j}$  and integrate over $\Omega$. This yields
\begin{equation}\label{po5}
\int_{\Omega} u_t(x \cdot \nabla u) \mathrm{d} x-\int_{\Omega} \Delta u(x \cdot \nabla u) \mathrm{d} x=\int_{\Omega} u^p(x \cdot \nabla u) \mathrm{d} x .
\end{equation}
Then the  divergence theorem implies that
\begin{align*}
0 & =\int_{\partial \Omega}\left [x \frac{1}{p+1}u^{p+1}\right] \cdot \nu \mathrm{d} S =\int_{\Omega} \sum_{i=1}^n\left[x_i \frac{1}{p+1}u^{p+1}\right]_{x_i} \mathrm{~d} x \\
& =N \int_{\Omega}\frac{1}{p+1}u^{p+1} \mathrm{d} x+\int_{\Omega}(x \cdot \nabla u) u^p \mathrm{d} x,
\end{align*}
which implies
\begin{equation}\label{po6}
\int_{\Omega}(x \cdot \nabla u) u^p \mathrm{d} x=-N\int_{\Omega} \frac{1}{p+1}u^{p+1} \mathrm{d} x .
\end{equation}
Since $u=0$ on $\partial \Omega$, the gradient $\nabla u$ is parallel to the unit outward normal $\nu$. Consequently,
\begin{equation}\label{po7}
 \nabla u=\partial_\nu u \nu, \quad(\nabla u \cdot \nu)(x \cdot \nabla u)=\left(\partial_\nu u\right)^2(x \cdot \nu).
\end{equation}
Using \eqref{po7} and Gauss divergence theorem, we obtain
\begin{align} \label{po8}
\int_{\Omega} \Delta u(x \cdot \nabla u) \mathrm{d} x&=\int_{\partial \Omega} \sum_{i, j=1}^N u_{x_i} \nu^i x_j u_{x_j} \mathrm{~d} S-\int_{\Omega} \sum_{i, j=1}^N u_{x_i}\left(x_j u_{x_j}\right)_{x_i} \mathrm{~d} x\nonumber\\
& =\int_{\partial \Omega}|\nabla u|^2(x \cdot \nu) \mathrm{d} S-\int_{\Omega} \sum_{i, j=1}^N u_{x_i} \delta_{i j} u_{x_j} \mathrm{~d} x-\int_{\Omega} \sum_{i, j=1}^N u_{x_i} x_j u_{x_j x_i} \mathrm{~d} x \nonumber\\
& =\int_{\partial \Omega}|\nabla u|^2(x \cdot \nu) \mathrm{d} S-\int_{\Omega}|\nabla u|^2 \mathrm{~d} x-\int_{\Omega} \sum_{i, j=1}^N\left(\frac{|\nabla u|^2}{2}\right)_{x_j} x_j \mathrm{~d} x \nonumber\\
& =\int_{\partial \Omega}|\nabla u|^2(x \cdot \nu) \mathrm{d} S-\int_{\Omega}|\nabla u|^2 \mathrm{~d} x-\int_{\partial \Omega} \frac{|\nabla u|^2}{2}(x \cdot \nu) \mathrm{d} S+N \int_{\Omega} \frac{|\nabla u|^2}{2} \mathrm{~d} x \nonumber\\
& =\int_{\partial \Omega} \frac{|\nabla u|^2}{2}(x \cdot \nu) \mathrm{d} S+\frac{N-2}{2} \int_{\Omega}|\nabla u|^2 \mathrm{~d} x .
\end{align}
Combining (\ref{po5}), (\ref{po6})  and (\ref{po8}), we obtain
\begin{align} \label{po9}
\int_{\Omega} u_t(x \cdot \nabla u) \mathrm{d} x & =\int_{\Omega} \Delta u(x \cdot \nabla u) \mathrm{d} x+\int_{\Omega} u^p(x \cdot \nabla u) \mathrm{d} x \nonumber\\
= & \int_{\partial \Omega} \frac{|\nabla u|^2}{2}(x \cdot \nu) \mathrm{d} S+\frac{N-2}{2} \int_{\Omega}|\nabla u|^2 \mathrm{~d} x
-N \int_{\Omega} \frac{1}{p+1}u^{p+1} \mathrm{d} x.
\end{align}
From \eqref{po4} and \eqref{po9}, we have  \eqref{po10}.

\end{proof}

\begin{proof}[Proof of Theorem \ref{t82}]
Since $\Omega$ is a strictly star-shaped with respect to the origin,
 it shows that $x\cdot\nu> 0$ for $x \in \partial\Omega$.
If (\ref{po11}) holds, we will derive a contradiction with (\ref{po10}).

\end{proof}

\end{document}